\documentclass[reqno, 12pt]{amsart}

\usepackage{amsfonts, amsthm, amsmath, amssymb}
\usepackage{hyperref}
\hypersetup{colorlinks=false}

\usepackage[margin=1.5in]{geometry}

\usepackage{helvet}

\RequirePackage{mathrsfs} \let\mathcal\mathscr

\usepackage{hyperref}
\usepackage{dsfont}
\usepackage{upgreek}
\usepackage{mathabx,yfonts}
\usepackage{enumerate}

\numberwithin{equation}{section}

\newtheorem{theorem}{Theorem}[section] 
\newtheorem{lemma}[theorem]{Lemma}
\newtheorem{proposition}[theorem]{Proposition}

\theoremstyle{definition}
\newtheorem*{acknowledgements}{Acknowledgements}
\newtheorem{remark}[theorem]{Remark}

\newcommand{\eul}{\mathrm{e}}

\renewcommand{\phi}{\varphi}

\newcommand{\0}{\mathbf{0}}

\newcommand{\PP}{\mathbb{P}}

\newcommand{\ZZ}{\mathbb{Z}}
\newcommand{\ZZp}{\mathbb{Z}_{\mathrm{prim}}}
\newcommand{\NN}{\mathbb{N}}
\newcommand{\QQ}{\mathbb{Q}}
\newcommand{\RR}{\mathbb{R}}
\newcommand{\CC}{\mathbb{C}}

\newcommand{\cA}{\mathcal{A}}

\newcommand{\cD}{\mathcal{D}}
\newcommand{\cM}{\mathcal{M}}

\newcommand{\cP}{\mathcal{P}}

\renewcommand{\leq}{\leqslant}

\renewcommand{\geq}{\geqslant}

\renewcommand{\bar}{\overline}

\newcommand{\x}{\mathbf{x}}
\newcommand{\y}{\mathbf{y}}
\renewcommand{\c}{\mathbf{c}}

\renewcommand{\v}{\mathbf{v}}

\newcommand{\w}{\mathbf{w}}
\renewcommand{\b}{\mathbf{b}}

\renewcommand{\r}{\mathbf{r}}

\newcommand{\fo}{\mathfrak{o}}
\newcommand{\fa}{\mathfrak{a}}
\newcommand{\fb}{\mathfrak{b}}
\newcommand{\fc}{\mathfrak{c}}
\newcommand{\fd}{\mathfrak{d}}
\newcommand{\fe}{\mathfrak{e}}

\newcommand{\fp}{\mathfrak{p}}

\newcommand{\ft}{\mathfrak{t}}
\newcommand{\fr}{\mathfrak{r}}
\newcommand{\fS}{\mathfrak{S}}
\newcommand{\fA}{\mathfrak{A}}

\newcommand{\ve}{\varepsilon}

\DeclareMathOperator{\rank}{rank}
\DeclareMathOperator{\disc}{disc}

\DeclareMathOperator{\Pic}{Pic}

\DeclareMathOperator{\vol}{vol}

\DeclareMathOperator{\res}{Res}

\DeclareMathOperator{\n}{N}

\DeclareMathOperator{\moo}{mod} 
\renewcommand{\bmod}[1]{\,(\moo{#1})}

\newcommand{\Zp}{\mathbb{Z}_{\text{prim}}}

\let\emptyset\varnothing

\DeclareSymbolFont{bbold}{U}{bbold}{m}{n}
\DeclareSymbolFontAlphabet{\mathbbold}{bbold}

\newcommand{\md}[1]{  \left(\textnormal{mod}\ #1\right)}
\newcommand{\lab}{\label} 

\newcommand{\Q}{\mathbb{Q}}

\newcommand{\N}{\mathbb{N}}
\newcommand{\R}{\mathbb{R}}

\newcommand{\Z}{\mathbb{Z}}
\renewcommand{\l}{\left}

\renewcommand{\r}{\right}
\renewcommand{\b}{\mathbf}
\renewcommand{\c}{\mathcal}
\renewcommand{\epsilon}{\varepsilon}

\renewcommand{\leq}{\leqslant}
\renewcommand{\geq}{\geqslant}
\renewcommand{\#}{\sharp}

\title[Rational points on quartic del Pezzo surfaces]{Counting rational points on quartic del Pezzo surfaces with a rational conic}

\author{T. D. Browning}
\address{School of Mathematics\\
University of Bristol\\ Bristol\\ BS8 1TW\\ UK}
\email{t.d.browning@bristol.ac.uk}

\author{E. Sofos}
\address{Max Planck Institute for Mathematics \\ 
Vivatgasse 7\\
Bonn \\ 
53111 \\ 
Germany} 
\email{sofos@mpim-bonn.mpg.de}

\subjclass[2010]{11G35 (11G50, 14G05)}
\date{\today}

\begin{document}

\begin{abstract}
Upper and lower bounds, of the expected order of magnitude,
are obtained for the number of rational points of bounded height on any quartic del Pezzo  surface over $\QQ$ that contains a conic defined over $\QQ$.  
\end{abstract}

\maketitle

\setcounter{tocdepth}{1}
\tableofcontents
 
\section{Introduction}\label{s:intro}

A quartic del Pezzo surface $X$ over $\QQ$ is a smooth projective surface
in $\PP^4$ cut out by a pair of quadrics  
 defined over $\QQ$.  When  $X$ contains a conic defined over $\QQ$ it may be equipped with a 
 dominant  $\QQ$-morphism
$X\to \PP^1$,  all of whose fibres are conics, giving $X$ the structure of a conic 
bundle surface. Let $U\subset X$ be the  Zariski open set obtained by deleting the $16$ 
lines
from $X$ and consider the counting function
$$
N(B) = \# \{ x \in U(\QQ): H(x) \leq B  \} , 
$$ 
for $B\geq 1$, 
where $H$ is the standard height function on $\PP^{4}(\QQ)$. 
The Batyrev--Manin conjecture \cite{f-m-t} predicts the existence of a  constant $c\geq 0$ such that 
$N(B)\sim cB(\log B)^{\rho-1}$, as $B\to \infty$, where $\rho=\rank \Pic_{\QQ}(X)\leq 6$. 
To date,  as worked out by  de la Bret\`eche and Browning \cite{dp4-smooth},
the only example for which this conjecture has been settled is the surface
$$
x_0x_1-x_2x_3=x_0^2+x_1^2+x_2^2-x_3^2-2x_4^2=0,
$$
with Picard rank $\rho=5$.
For a general quartic del Pezzo surface the best upper bound we have is 
$N(B) =O_{\ve,X}(B^{\frac{3}{2}+\ve}),$ for any $\ve>0$,
which appears in forthcoming work of
Salberger.

In work presented at the conference ``Higher dimensional varieties and rational points'' at Budapest in 2001, Salberger noticed that one can get much better upper bounds for 
$N(B)$ when $X$ has a conic bundle structure over $\QQ$, ultimately showing   that 
$N(B)=O_{\ve,X}( B^{1 + \ve})$,
 for all $\ve>0$.
Leung \cite{fok} revisited Salberger's argument to promote the $B^\ve$ to an explicit power of $\log B$.
On the other  hand,  recent work of Frei, Loughran and Sofos
\cite[Thm.~1.2]{FS}   
provides
a lower bound for $N(B)$ of the predicted
order of magnitude  for any quartic del Pezzo surface over $\QQ$ with a $\QQ$-conic bundle structure and  Picard rank $\rho\geq 4$.
(In fact they have results over any number field and for conic bundle surfaces of any degree.)
Our main result goes further and shows that the expected  upper and lower bounds can be obtained for any conic bundle quartic del Pezzo surface over $\QQ$. 
\begin{theorem}
\label{th:upper}
Let $X$ be a  quartic del Pezzo surface defined over $\Q$, such that
$X(\QQ)\neq \emptyset$. 
If $X$ contains a conic defined over $\Q$ 
then there exist 
effectively computable 
constants $c_1,c_2, B_0>0$, depending on $X$, such that for all $B\geq B_0$ we have
$$
c_1B (\log B)^{\rho-1} \leq N(B)\leq c_2 B (\log B)^{\rho-1}.
$$
\end{theorem} 
It is worth emphasising that this appears to be the first time that 
sharp bounds are achieved 
towards the Batyrev--Manin conjecture for del Pezzo surfaces 
that are not necessarily rational over $\QQ$. 

Let $X$ be a  quartic del Pezzo surface 
defined over $\Q$, with a conic bundle structure
$\pi: X\to \PP^1$. 
There are $4$ degenerate geometric  
fibres of $\pi$ and it follows from work of 
Colliot-Th\'el\`ene \cite{ct-4} and Salberger \cite{s-conic}, using 
independent approaches, that 
the Brauer--Manin obstruction is the only obstruction to the Hasse principle and weak approximation. 
Let $\delta_0\leq \delta_1\leq 4$, where $\delta_1$ is the number of closed points in $\PP^1$ above which $\pi$ is degenerate and 
$\delta_0$ is the number of these with split fibres.
(Recall from \cite[Def.~0.1]{S96} that a scheme over $\QQ$ is called {\em split} if it contains 
a non-empty geometrically integral open subscheme.)
It follows from  \cite[Lemma 2.2]{FS} that
\begin{equation}\label{eq:rank}
\rho=2+\delta_0.
\end{equation}
For comparison,  Leung's work \cite[Chapter~4]{fok} establishes an upper bound for $N(B)$ with the potentially larger exponent $1+\delta_1$.
This exponent agrees with the Batyrev--Manin conjecture if and only if $X\to \PP^1$ is
a conic bundle with a section over $\QQ$, a hypothesis that our main result 
avoids. 

Our proof of the upper bound makes essential use of~\cite{sofos},  
where detector functions are worked out for the fibres with $\QQ$-rational points.
Combining this with height machinery and a  uniform estimate \cite{n-2} for the number of rational points of bounded height on a conic, the  problem is reduced  to  finding optimal upper bounds  for
divisor sums of the shape
\begin{equation}\label{eq:goat}
\sum_{\substack{(s,t) \in \Z^2\\ \max\{|s|,|t|\}\leq x }}
\prod_{i=1}^n
\sum_{\substack{d_i \mid \Delta_i(s,t)}}
\left(
\frac{G_i(s,t)}{d_i}
\right).
\end{equation}
Here,  $n=\delta_1$ and 
$\Delta_1,\dots,\Delta_n\in \Z[s,t]$ 
are the closed points of $\PP^1$ above which $\pi$ is degenerate, 
with  $G_1,\dots,G_n\in \Z[s,t]$ being certain associated forms of even degree.
Thus far, such sums have only been examined in the special case that 
$G_1,\dots,G_n$ all have degree zero.  
In this setting, 
work of la Bret\`eche and Browning \cite{nair}
 can be invoked to yield the desired upper bound. Unfortunately, this result 
 is no longer applicable when one of $G_1,\dots,G_n$ has positive degree.

Using~\cite{FS}, we shall see in \S \ref{s:lower} that  
our proof of the lower bound in Theorem~\ref{th:upper} may proceed for surfaces $X\to \PP^1$ of Picard rank $\rho=2$. 
In this case  the fibre above any  degenerate  closed point of $\PP^1$ must be non-split by \eqref{eq:rank}.
Ultimately, following the strategy of \cite{FS}, this leads to the problem of proving tight lower bounds for sums like \eqref{eq:goat} in the special case that none of the characters $(\frac{G_i(s,t)}{\cdot})$ are trivial.  One of the key ingredients in this endeavour is a generalised Hooley $\Delta$-function.  
Let $K/\QQ$ be a number field
and let $\psi_K$ be a  quadratic 
Dirichlet character on $K$. We define an arithmetic function 
on integral ideals of $K$ via
$$
\Delta(\fa;\psi_K)=
\sup_{\substack{u \in \R\\ 0\leq v\leq 1}}
\Big|
\sum_{\substack{
\fd\mid \fa \\
\eul^u<\n_K\fd\leq \eul^{u+v} 
}}
\psi_K(\mathfrak{d})
\Big|,
$$
for any ideal $\fa$ in the ring of integers $\fo_K$ of  $K$, where $\n_K $ denotes the ideal norm.
When $K=\QQ$ this recovers the twisted  $\Delta$-function considered  
by  la Bret\`eche--Tenenbaum \cite{daniel-F} and Br\"udern \cite{daniel-D}.  
Our treatment of the lower bound requires a 
second moment estimate for 
$\Delta(\fa;\psi_K)$ and this is supplied in a companion paper of Sofos \cite{sofos4hooley}.

\begin{remark}
Ch\^atelet surfaces  provide the other family  of  relatively minimal conic bundle surfaces
of degree $4$.  When they are defined over $\QQ$,   the Batyrev--Manin conjecture also makes a prediction for the distribution of $\QQ$-rational points on them.  Work of Browning  \cite{chatelet} shows that the relevant counting function satisfies an upper bound of  the expected size.
Although we shall not provide any details here, if we 
suppose that the Ch\^atelet surface has a $\QQ$-rational point,
then a  lower bound of the proper size follows from the work in this paper, on taking the forms $G_1,\dots,G_n$ to have degree $0$ in \eqref{eq:goat}. 
\end{remark}

The main novelty in our work lies in how we overcome the difficulty of divisor sums involving characters without a fixed modulus in \eqref{eq:goat}.
In~\S\ref{s:sausage}, drawing inspiration from recent work of Reuss~\cite{reuss}, we replace the divisor functions
at hand by generalised divisor functions  which run over certain integral ideal  divisors belonging to the number field obtained by adjoining a root of $\Delta_i$, for each $1\leq i\leq n$.   
Our proof of Theorem \ref{th:upper}  then relies upon an extension to number fields of work by Nair and Tenenbaum \cite{NT} on short sums of non-negative  arithmetic functions.   This is achieved in an auxiliary investigation \cite{ANT-NT}, the outcome of which is recorded in \S \ref{s:Mk}.

\begin{acknowledgements}
We are very grateful to Roger Heath-Brown for useful discussions 
and to the anonymous referee for numerous helpful comments that have clarified the exposition considerably. 
While working on this paper the first author was
supported by ERC grant \texttt{306457}. 
\end{acknowledgements}

\section{Preliminary results}

\subsection{Nair--Tenenbaum over number fields}\label{s:Mk}

Let $K/\QQ$ be a number field and let $\fo_K$ be its ring of integers.
Denote by $\c{I}_K$ the set of ideals in $\fo_K$. 
We say that a function $f:\c{I}_K\to \RR_{\geq 0}$ 
is {\em pseudomultiplicative} 
if there exist strictly positive constants $A,B,\ve$ such that   
$$
f(\fa \fb)
\leq 
f(\fa)
\min\left\{A^{\Omega_K(\fb)}, B(\n_K\fb)^\ve\right\},
$$
for all coprime ideals $\fa,\fb\in \c{I}_K$, where
$
\Omega_{K}(\fb)=\sum_{ \fp \mid  \fb}
\nu_\fp(\fb).
$
We denote the class of all pseudomultiplicative functions associated to $A,B$ and $\ve$  by  $\cM_K=\cM_K(A,B,\ve)$. 
Note that any $f\in \c{M}_K$ satisfies the bounds
$f(\fa) \ll A^{\Omega_K(\fa)}$ and 
$f(\fa) \ll
(\n_K \fa)^\varepsilon$,  for any  $\fa\in \c{I}_K$.

We will need to work with functions supported away from  ideals of small norm.  To facilitate this, for any ideal $\fa\in \c{I}_K$ and $W\in \NN$, we set
\begin{equation}\label{eq:aW}
\fa_W=\prod_{\substack{\fp^\nu \| \fa \\ \gcd(\n_K \fp,W)=1}}
\fp^\nu.
\end{equation}
We extend this to rational integers in the obvious way. 
Similarly,  for any $f\in \c{M}_K$, we define  
$
f_W(\fa)=f(\fa_W).
$
\begin{remark} 
We will always assume that $W$ is of the form
\begin{equation}\label{eq:WW}
W=\prod_{p\leq w} p^\nu,
\end{equation}
for some $w>0$ 
and $\nu$ a positive integer. Throughout~\S\ref{s:lower} we shall take $\nu$ to be a large constant depending only
on various polynomials that are determined  by $X$,  while in~\S\ref{s:epikoskafes} we shall take $\nu=1$. 
In either case we have  $\gcd(\n_K\fp,W)=1$ if and only if $p>w$, if $\n_K\fp=p^{f_\fp}$ for some $f_\fp\in \NN$.
Our notation is reminiscent of the ``$W$-trick'' that appears in work of Green and Tao \cite{GT}. Whereas in their context it is important that the parameter $w$ tends to infinity, in our setting we shall choose $w$ to be a suitably large constant, where the meaning of ``suitably large'' is allowed to change at various points of the proof.
\end{remark}

Let
\begin{equation}\label{eq:span-0}
\!\cP_K^\circ\!=\!
\left\{\fa\subset \fo_K: 
\fp\mid \fa \Rightarrow f_{\fp}=1
\right\}
\end{equation}
be the
multiplicative span of all prime ideals $\fp\subset \fo_K$ with residue degree $f_\fp=1$. 
For any  $x>0$
and  $f\in \c{M}_K$ we 
 set  
$$
E_{f}(x;W)=\exp\Bigg(
\sum_{\substack{
\fp\in \cP_K^\circ \text{ prime}\\
w<\n_K 
\fp \leq x
\\
f_\fp=1 }}\frac{f(\fp)}{\n_K\fp}\Bigg),
$$
if $f$ is  submultiplicative, and 
$$
E_{f}(x;W)=
\sum_{\substack{
\n_K\fa  \leq x\\
\fa\in \cP_K^\circ  \text{ square-free} \\
\gcd(\n_K \fa,W)=1
}}\frac{f(\fa)}{\n_K\fa},
$$
otherwise.

Suppose now that we are given irreducible binary forms $F_1,\dots,F_N\in \ZZ[x,y]$, which we assume to be  pairwise coprime. 
Let $i\in \{1,\dots,N\}$. Suppose that $F_i$ has degree $d_i$
and that it is not proportional to $y$, so that  $b_i=F_i(1,0)$ is a non-zero integer.
It will be convenient to form the homogeneous polynomial
\begin{equation}\label{eq:burger}
\tilde F_i(x,y)=b_i^{d_i-1}F_i(b_i^{-1}x,y).
\end{equation}
This has integer coefficients and satisfies $\tilde F_i(1,0)=1$.
We  let $\theta_i$ be a root of the monic polynomial $\tilde F_i(x,1)$.
Then $\theta_i$ is an algebraic integer and we denote the  associated number field of degree $d_i$ by 
   $K_i=\QQ(\theta_i)$.
Moreover,  
$$N_{K_i/\QQ}(b_is-\theta_it)=\tilde F_i(b_is,t)=b_i^{d_i-1}F_i(s,t),$$
for any $(s,t)\in \ZZ^2$.
If  $b_i=0$, so that  $F_i(x,y)=c y$ for some non-zero $c \in \ZZ$,
we take  $\theta_i=-c$ and $K_i=\Q$ in this discussion.
Our work on Theorem \ref{th:upper} requires tight upper bounds for averages of 
$f_{1,W}((b_1s-\theta_1 t))\dots f_{N,W}((b_Ns-\theta_N t))$, over primitive vectors  $(s,t)\in \ZZ^2$, 
for general pseudomultiplicative functions $f_i\in \c{M}_{K_i}$ and suitably large  $w$. 

For any $k\in \NN$ and any polynomial $P\in \ZZ[x]$, we set 
\begin{equation}\label{eq:def-rho}
\rho_{P}(k)=\#\{x\bmod{k} : P(x)\equiv 0 \bmod{k}\}.
\end{equation}
Let 
$\bar\rho_i(k)=
\rho_{F_i(x,1)}(k)$  if $F_i(1,0)\neq 0$ and 
$\bar\rho_i(k)=1$ if $F_i(1,0)= 0$.
Moreover, put
\begin{equation}\label{eq:faggot}
h^*(k)=\prod_{p\mid k} \left(1-\frac{\bar\rho_1(p)+\dots +\bar\rho_N(p)}{p+1}\right)^{-1}.
\end{equation} 
To  any non-empty  bounded measurable  region
$\c{R} \subset \R^2$, 
we associate
\begin{equation*}
K_\c{R}=1+
\|\c{R}\|_{\infty}
+\partial(\c{R})
\log(1+\|\c{R}\|_\infty)
+\frac{\mathrm{vol}(\c{R})}{1+\|\c{R}\|_\infty}
,\end{equation*}
where $\|\c{R}\|_\infty= \sup_{(x,y) \in \c{R}}\{ |x|,|y|\}$.
We say that such a region $\c{R}$ is {\em regular} if 
its boundary  is piecewise differentiable,  $\c{R}$ contains no 
zeros of 
$F_1\cdots F_N$  and 
 there exists
$c_1>0$ such that 
$\mathrm{vol}(\c{R})\geq K_\c{R}^{c_1}$. 
Bearing all of this in mind, the following result is  \cite[Thm.~1.1]{ANT-NT}. 

\begin{lemma}\label{t:NT}
Let $\c{R} \subset \R^2$ be a regular region,
let    $V=\mathrm{vol}(\c{R})$ 
and let $G \subset \Z^2$ be a lattice of full rank, with  determinant $q_G$ and  first successive minimum  $\lambda_G$.
Assume that  $q_G \leq V^{c_2}$ for some $c_2>0$.
Let  $f_i \in \c{M}_{K_i}(A_i,B_i,\ve_i
)$ for $1\leq i\leq N$
and let 
$$\ve_0= 
\max\bigg\{1+\frac{4}{c_1},\frac{4(5+3
\max\{\ve_1,\dots,\ve_N\}) 
}{c_1}\bigg\}
\bigg(\sum_{i=1}^N d_i \ve_i\bigg)
.$$ 
Then, for any $\ve>0$ and $w>w_0(f_i,F_i,N)$, we have 
\begin{align*}
\sum_{(s,t) \in \Zp^2 \cap \c{R}\cap G}
\prod_{i=1}^N
f_{i,Wq_G}((b_i s -\theta_i t))
\hspace{-0.08cm}
\ll~&
\frac{V}{(\log V)^N}
\frac{h_W^*(q_G)
}{q_G}
\prod_{i=1}^N
E_{f_i}(V;W)\\
&+\frac{K_{\c{R}}^{1+\ve_0+\ve}}{\lambda_G} 
,\end{align*}
where the implied constant depends at most on
$c_1,c_2,A_i,B_i,F_i, \ve,  
\ve_i, N,W$.
\end{lemma}

Let $1\leq i\leq n$.
In the statement of this result we recall the convention that 
the function $f_{i,W q_G}$ is defined in such a way that 
$f_{i,Wq_G}(\fa)=f_{i}(\fa_{W q_G})$ 
for  any integral ideal $\fa\subset \fo_{K_i}$, where
\[
\fa_{Wq_G}=
\prod_{\substack{\fp^\nu \| \fa \\ \gcd(\n_K \fp,W)=1
\\ \gcd(\n_K \fp,q_G)=1
}}
\fp^\nu
.\]

\subsection{Divisor sums over number fields}
\label{s:sausage}

Let  $K/\QQ$ be a finite extension of degree
$d$. We write  $\fo=\fo_K$ and $\n=\n_K$ for the 
ring of integers and ideal norm, respectively. 
Let   $\sigma_1,\dots,\sigma_d:K \hookrightarrow\CC$
be the associated 
embeddings and let  $\{\omega_1,\dots,\omega_d\}$
be a $\ZZ$-basis for $\fo$.
Let  $\fa\subset \fo$ be an integral ideal 
with  $\ZZ$-basis $\{\alpha_1,\dots,\alpha_d\}$. 
We henceforth set $\Delta(\alpha_1,\dots,\alpha_d)=|\det(\sigma_i(\alpha_j))|^2$,
and similarly for $\{\omega_1,\dots,\omega_d\}$.
According to \cite[Satz 103]{landau_alg}, we have 
\begin{equation}\label{eq:sainsburys}
\Delta(\alpha_1,\dots,\alpha_d)=(\n \fa)^2 D_K,
\end{equation} 
where $D_K=\Delta(\omega_1,\dots,\omega_d)$ is the  discriminant of $K$.

Let $F,G\in \ZZ[x,y]$ be non-zero binary forms with $F$ irreducible,
$G$ of even degree and non-zero resultant $\mathrm{Res}(F,G)$.
We shall assume that $F$ has degree $d$
and that it is not proportional to $y$.
In particular $b=F(1,0)$ is a non-zero integer.
Let $W\in \NN$.
For any $(s,t)\in \ZZp^2$ such that $F(s,t)\neq 0$,
we  define 
\begin{equation}\label{eq:def-hW}
h_W(s,t)=\sum_{\substack{k \mid F(s,t)\\ \gcd(k,W)=1}}\left(\frac{G(s,t)}{k}\right).
\end{equation}
This is a  modified version of the functions that appear in \eqref{eq:goat}.
We recall
from \eqref{eq:burger} the   associated binary form
$\tilde F(x,y)=b^{d-1}F(b^{-1}x,y)$,
with  integer coefficients and $\tilde F(1,0)=1$.  We conclude that for all non-zero integer multiples $c$ of $b$, we have 
$$
h_{cW}(s,t)=\sum_{\substack{k \mid \tilde F(bs,t)\\ \gcd(k,cW)=1}}\left(\frac{G(s,t)}{k}\right),
$$
since $k\mid \tilde F(bs,t) $ if and only if $k\mid F(s,t)$.

We henceforth let $\theta$ be a root of the polynomial $f(x)=\tilde F(x,1)$.
Then $\theta$ is an algebraic integer and   $K=\QQ(\theta)$ is a number field of degree $d$ over $\QQ$.   
It follows that 
$\ZZ[\theta]\subset \fo$ is an order
of $K$ with  discriminant 
$\Delta_\theta=\Delta(1,\theta,\dots,\theta^{d-1})$. In view of \eqref{eq:sainsburys} we have 
\begin{equation}\label{eq:tesco}
\Delta_\theta=[\fo:\ZZ[\theta]]^2D_K.
\end{equation}

We  now let $L=K(\sqrt{g(\theta)})$, 
where $g(x)=G(b^{-1}x,1)\in \QQ[x]$.
We shall assume that $L/K$ is a quadratic extension and we let $D_{L/K}$ be the ideal norm
of the relative discriminant $\mathfrak{D}_{L/K}$.  
Let  $\mathfrak{f}=\mathfrak{f}_{L/K}$
be the conductor of the extension $L/K$.
 Let $J^\mathfrak{f}$ be the group of fractional ideals in $K$ coprime to $\mathfrak{f}$
and let $P^\mathfrak{f}$ be the group of principal ideals
$(a)$ such that $a\equiv 1\bmod{\mathfrak{f}}$ and $a$ totally positive.
As explained by Neukirch \cite[\S VII.10]{neukirch}, 
the Artin symbol $\psi(\fa)=(\frac{L/K}{\fa})$ gives rise to a character $
\psi:J^\mathfrak{f}/P^\mathfrak{f}\rightarrow\{\pm 1\}
$
of the ray class group 
$J^\mathfrak{f}/P^\mathfrak{f}$, with $\fa\bmod{P^{\mathfrak{f}}}\mapsto (\frac{L/K}{\fa})$. This has the property that $\psi(\fp)=1$ 
if and only if $\fp$ splits in $L$, for any unramified prime  ideal $\mathfrak{p}\in J^\mathfrak{f}$.

Let 
\begin{equation}\label{eq:D}
D=2b D_{L/K} \Delta_\theta  
\n \mathfrak{f}.  
\end{equation}
Note that $D$ is a non-zero integer.  
Recall the definition \eqref{eq:span-0} of $\cP_K^\circ$  
of the multiplicative span  of degree $1$ prime ideals. 
We shall mainly  work with the subset
\begin{equation}\label{eq:span}
\!\cP_K\!=\!
\left\{\fa\subset \cP_K^\circ: 
\fp_1\fp_2\mid \fa \Rightarrow
\n_K\fp_1\neq 
\n_K\fp_2
 \text{ or }
\fp_1= \fp_2
\right\}
\end{equation}
cut out by ideals 
divisible by at most one prime ideal above each rational prime.
It is not hard to see that $\cP_K$ has positive density in $\mathcal{I}_K$.  
The proof of the following result is inspired by an argument found in recent work of Reuss \cite[Lemma 4]{reuss}.

\begin{lemma}\label{lem:reuss}
Let $W\in \NN$, 
let $(s,t)\in \Zp^2$ such that  $F(s,t) \neq 0$, and let 
$D$ be given by \eqref{eq:D}. 
Then the following hold:
\begin{itemize}
\item[(i)]
$\fa\in \cP_K$ for any integral ideal $\fa\mid (bs-\theta t)$
such that $\gcd(\n\fa,DW)=1$;
\item[(ii)]
there exists  a bijection between 
divisors $\fa\mid (bs-\theta t)$ with $\n \fa=k$ coprime to $DW$ and divisors $k\mid \tilde  F(bs,t)$ coprime to $DW$,  in which 
$\Omega(k)=\Omega_K(\fa)$ and $(\frac{G(s,t)}{k})=\psi(\fa)$;
\item[(iii)]
we have 
$$
h_{DW}(s,t)=\sum_{\substack{\fa\mid (bs-\theta t)\\ \gcd(\n\fa,DW)=1}} \psi(\fa).
$$
\end{itemize}
\end{lemma}

In particular, when $G(s,t)$ is the constant polynomial $1$ in  
\eqref{eq:def-hW}, then $L=K$ and $\psi$ is just the trivial character in part (iii).
We note that  $\Omega_K(\fa)=\Omega(\n \fa)$ 
and $\tau_K(\fa)=\tau(\n\fa)$ 
for any ideal $\fa\in \mathcal{P}_K$, 
where $\tau_K(\fa)=\sum_{\fd\mid \fa}1$.   
Similarly, if 
$h:\NN\to \RR_{\geq 0}$
 is any arithmetic function, we have 
$$
 \prod_{\fp\mid \fa} \left(1+h(\n\fp)\right)=\prod_{p\mid \n\fa} \left(1+h(p)\right),  
$$
for any $\fa\in \cP_K$.
 We shall use these facts without further comment in the remainder of the paper.

\begin{proof}[Proof of Lemma \ref{lem:reuss}]
Let $(s,t)\in \Zp^2$ such that  $F(s,t) \neq 0$. We form the integral ideal  
$\mathfrak{n}=(bs-\theta t).$
This has norm $\n \mathfrak{n}=|\tilde F(bs,t)|.$
Let $k\mid \tilde F(bs,t)$ with $\gcd(k,DW)=1$.   
In particular   $\gcd(k,\Delta_\theta)=1$. 

Part (i) is proved in  
\cite[Lemma~2.3]{ANT-NT}.  
Turning to part (ii), it follows from (i)  that  $(p,\mathfrak{n})$ is a prime ideal for any $p\mid k$. Thus there is a bijection between each factorisation
 $|\tilde F(bs,t)|=ke$, with $\gcd(k,DW)=1$,  and  each ideal factorisation 
$\mathfrak{n}=\mathfrak{a}\fb$, with
 $\n\fa=k$ coprime to $DW$ and $\n\fb=e$.	
 In order to complete the proof of part (ii) of the lemma, it will suffice to show that 
\begin{align*}
\left(\frac{G(s,t)}{p}\right)=\psi(\mathfrak{p}),
\end{align*}
where $\fp=(p,\mathfrak{n})$. 
Since $G$ has even degree we have 
$$
\left(\frac{G(s,t)}{p}\right)=\left(\frac{G(s\bar t,1)}{p}\right).
$$
Recall the notation $g(x)=G(b^{-1}x,1)$. 
We may suppose  that $\fp=(p,\theta-n)$, for some $n\in \ZZ/p\ZZ$
such that 
$bs\bar{t}-n\equiv 0\bmod{p}$, and we
 recall from \eqref{eq:D} that $p\nmid 2D_{L/K}$. 
We observe that  $\mathfrak{p}$ splits in $L=K(\sqrt{g(\theta)})$ if and only if $g(n)$ is a square in $\fo/\fp$, since $g(\theta)\equiv g(n)\bmod{\fp}.$
 But this is if and only if 
$$
\left(\frac{g(bs\bar t)}{p}\right)=1,
  $$
since $n\equiv bs\bar t\bmod{p}$ and $\n\fp=p$.
Noting that $g(bs\bar t)=G(s\bar t,1)$, this 
completes the proof of part (ii). Finally, part (iii) follows from part (ii).
\end{proof}

We close this section with an observation about the condition 
$\mathfrak{a}
\mid (b s-\theta t)$ that appears in 
Lemma \ref{lem:reuss}, the proof of which is 
found in  
 \cite[Lemma 2.4]{ANT-NT}.   
  
\begin{lemma}\label{lem:honey}  
Let $\fa\in \cP_K$ 
such that $\gcd(\n\fa,D_K)=1$.
Then
there exists 
$k=k(\fa)\in \ZZ$ such that 
$
\fa\mid (bs-\theta t) \Leftrightarrow 
bs\equiv k t\bmod{\n \fa}
$,
for all $(s,t)\in \ZZ^2$. 
\end{lemma}

\subsection{Uniform upper bounds for  conics}

Let $Q\in \ZZ[y_1,y_2,y_3]$ be a non-singular isotropic quadratic form.
Denote
its discriminant by
$\Delta_Q$ and the greatest common divisor of the $2\times 2$  minors of the associated matrix 
by
$D_Q$. 
It follows from \cite[\S IV.2]{schmidt} that there is a quadratic Dirichlet 
character $\chi_Q$ such that 
\[
\#\{\y\bmod{p}:Q(\y)\equiv 0 \bmod{p}, ~p\nmid\y\} 
=p(p-1)\left(1+\chi_Q(p)\right)+p-1,
\]
for any prime $p$ such that $p\mid \Delta_Q$ and $p\nmid 2 D_Q$.

The main aim of this section is to establish the following result.

\begin{lemma}
\label{detectors 2}
 Let  $w,B_1,B_2,B_3>0$ be given. 
Then 
\[
\#\left\{\y \in \ZZp^3: Q(\y)=0, ~|y_i|\leq B_i 
\right\}
\ll C(Q,w)
\left( 1+\frac{\l(B_1B_2B_3\r)^{\frac{1}{3}}D_Q^{\frac{1}{2}}} 
{|\Delta_Q|^{\frac{1}{3}}
}\right),
\]
with an absolute implied constant, 
where
\[
C(Q,w)=
\prod_{\substack{p^\xi \| \Delta_Q
\\
p \mid  2 D_Q \text{ or } p\leq w
}}
\tau(p^\xi )
\prod_{\substack{
p^\xi  \| \Delta_Q
\\ 
p>w
\\
p\nmid 2D_Q
}}
\l(
\sum_{k=0}^\xi 
\chi_Q(p)^k
\r).
\]
\end{lemma}

Since $C(Q,w)\leq \tau(\Delta_Q)$, this result is a refinement of work due to Browning  and Heath-Brown 
\cite[Cor.~2]{n-2}. 
In fact, although not needed  here,
one can show that for any prime $p\nmid 2D_Q$, the $p$-adic factor appearing above is commensurate 
with the  $p$-adic 
Hardy--Littlewood density for the conic $Q=0$. Furthermore, 
if this curve  has no $\QQ_p$-points for some prime $p\nmid 2D_Q$,
then the constant in the upper bound vanishes.
Therefore, Lemma~\ref{detectors 2} 
detects conics with a rational point.
This is the point of view adopted in the work of Sofos \cite{sofos}.
 
\begin{proof}[Proof of Lemma~\ref{detectors 2}]
The proof of 
\cite[Cor.~2]{n-2} relies on earlier work of Heath-Brown \cite[Thm.~2]{cubic}.
The latter work produces an upper bound for the number  of  lattices (with determinant depending on the coefficients of $Q$) that any non-trivial zero of $Q$ is constrained to lie in. 
For each prime $p$ such that 
$p^\xi\| \Delta_Q$,  it turns out that there are at most $L(p^\xi)\leq c_p\tau(p^\xi)$ lattices to consider, where $c_p=1$ for $p>2$.

Suppose that $\y\in \ZZp^3$ is a non-zero vector for which $Q(\y)=0$.
Let $p$ be a prime such that $p^\xi\| \Delta_Q$, with $p\nmid 2 D_Q$ and $\chi_Q(p)=-1$.
On diagonalising over $\ZZ/p^{\xi+1}\ZZ$, we may assume that 
$$
a_1y_1^2+a_2y_2^2+p^\xi y_3^2\equiv 0 \bmod{p^{\xi+1}},
$$
for coefficients $a_1,a_2\in \ZZ$ such that  $p\nmid a_1a_2$. In particular, we have
$\chi_Q(p)=(\frac{-a_1a_2}{p})=-1$. Hence $L(p^\xi)=1$ when 
$\xi$ is even, since then 
$\y$ is merely  constrained to lie on the lattice $\{\y\in \ZZ^3: y_1\equiv y_2\equiv 0\bmod{p^{\xi/2}}\}$.
Likewise, when $\xi$ is odd, there can be no  solutions in primitive integers  $\y$.

Note that 
\[
\sum_{k=0}^\xi 
\chi_Q(p)^k=\begin{cases}
\tau(p^\xi ) &\text{ if $\chi_Q(p)=1$},\\
1 &\text{ if $\chi_Q(p)=-1$ and $\xi$ is even},\\ 
0 &\text{ if $\chi_Q(p)=-1$ and $\xi$ is odd.}
\end{cases}
\]
It follows that the total number of lattices emerging is 
\begin{align*}
&\ll \mathbf{1}(\Delta_Q)
\prod_{\substack{p^\xi \| \Delta_Q
\\
p | 2 D_Q
}}
\tau(p^\xi )
\prod_{\substack{p^\xi \| \Delta_Q
\\
p \leq w
\\
p\nmid 2 D_Q
}}
\tau(p^\xi )
\prod_{\substack{p^\xi  \| \Delta_Q
\\
\chi_Q(p)=1 
\\
p>w
\\
p\nmid 2D_Q
}}
\tau(p^\xi ) 
=C(Q,w),
\end{align*}
where 
$\mathbf{1}(\Delta_Q)=0$
(resp.~$\mathbf{1}(\Delta_Q)=1$)
 if  there exists $p^\xi  \| \Delta_Q$ such that 
$\chi_Q(p)=-1$, with $\xi$ odd and 
$p\nmid 2D_Q$ (resp.~otherwise).
This completes the proof of the lemma.
\end{proof}

\subsection{Lattice point counting}
\label{s:doctorwhat}
We will need general results about counting lattice points in an expanding region. 
Let $\c{D}\subset \R^2\setminus \{\b{0}\}$ be a non-empty open disc
and put
$\delta(\cD)=\|\cD\|_\infty$, in the notation of \S \ref{s:Mk}.
Let $b,c,q \in \Z$ and $\x_0\in \ZZ^2$ such that  $q\geq 1$ and 
$\gcd(\x_0,q)=1$.
For each $e\in \NN$ such that 
$\gcd(e,q)=\gcd(b,c,e)=1$, we  define the non-empty
set 
\[
\Lambda(e)=\{(s,t)\in \Z^2:bs\equiv ct \md{e}\}.
\]
We then  fix, once and for all,
a non-zero
vector of minimal Euclidean length within $\Lambda(e)$
and we call it $\b{v}(e)$. We are interested in 
$$
N(x)=\#\Big\{\b{x}\in \Zp^2 \cap x\c{D}\cap \Lambda(e):\b{x}\equiv \x_0\md{q}\Big\},
$$
as $x\to \infty$.
We shall prove the following result.

\begin{lemma}
\lab{lem:historical}
Let $\cD, b,c,\x_0,q,\Lambda(e), \v(e), N(x)$ be as above, and assume that $|\b{v}(e)|
\leq \delta(\cD) x$.
Then 
\begin{align*}
N&(x)=\frac{\mathrm{vol}(\c{D})x^2}{\zeta(2) e q^2}
\prod_{p | e} 
\l(1+\frac{1}{p}\r)^{-1}
\prod_{p | q} 
\l(1-\frac{1}{p^2}\r)^{-1}\\
&+O\left(
\l(\beta+\gamma\r)
x
\l\{
\l(
\sum_{d|e}\frac{1}{d|\b{v}(e/d)|}
\log\left(2+ \frac{\delta(\cD) x}{d|\b{v}(e/d)|}\right) \r)  
+
\frac{1}{e}
\sum_{d|e}|\b{v}(d)|
\r\}\right),
\end{align*}
where
\[
\beta=\delta(\cD)+\frac{\partial{\c{D}}}{q}, \quad 
\gamma=
\frac{\mathrm{vol}(\c{D})}{\delta(\cD) q^2}
.\]
The  implied constant in this estimate is absolute. 
\end{lemma} 

For any $d\mid e$, let us denote
$\b{v}(e/d)$ by $(x_0,x_1)$, temporarily. Then 
\[
\frac{e}{d}
\mid 
(bx_0-c x_1)
\Rightarrow
(dx_0,dx_1) \in \Lambda(e),
\]
whence
\begin{equation}
\label{eq:painkiller}
|\b{v}(e)|
\leq 
d |\b{v}(e/d)|
.\end{equation}
Moreover, using the basic properties of the minimal basis vector,  one  obtains
\begin{equation}\label{eq:paingiver}
\frac{1}{e}\sum_{d\mid e} |\v(d)| \ll \frac{1}{e}\sum_{d\mid e} \sqrt{d}\leq \frac{\tau(e)}{\sqrt{e}}\ll \frac{\tau(e)}{|\v(e)|}.
\end{equation}
These inequalities may be used to 
simplify 
the error  term in  Lemma \ref{lem:historical}.

\begin{proof}[Proof of Lemma \ref{lem:historical}]
Our argument is based on a modification of the proof of 
  \cite[Lemma 5.3]{sofos}.
We write $\delta=\delta(\cD)$ for short and put 
$\x_0=(s_0,t_0)$. 
Since $\gcd(s_0,t_0,q)=1$, an application  of M\"obius inversion gives
\[
N(x)=
\sum_{\substack{m \in \N \\ \gcd(m,eq)=1}}
\mu(m)
\sum_{\substack{(u,v) \in \frac{x}{m}\c{D}\cap \Lambda(e)\\ 
\gcd(u,v,e)=1
\\
(u,v)\equiv \bar{m}(s_0,t_0)\md{q}}}
1
.\]
on making the  substitution $s=mu$ and $t=mv$.  
The inner sum is empty if $m$ is large enough.
Indeed, if it contains any terms then we must have 
$$
1\leq |\b{v}(e)|
=
\min\{|\b{y}|:\b{y} \in \Lambda(e)\setminus \{\0\}\}
 \leq \max\left\{|\b{y}|:\b{y} \in \frac{x}{m}\c{D}\right\}
\leq \frac{\delta x}{m}.
$$
Thus, on using the M\"obius function to remove the condition
$\gcd(u,v,e)=1$, 
we find that 
\[
N(x)=
\sum_{\substack{m \in \N \\ \gcd(m,eq)=1
\\
m\leq \frac{\delta x}{|\b{v}(e)|}
}}
\mu(m)
\sum_{d\mid e}\mu(d)
\sum_{\substack{(u,v) \in \frac{x}{m}\c{D}\cap \Lambda(e)
\\
d\mid u, ~d\mid v\\
(u,v)\equiv \overline{m} (s_0,t_0)\md{q}}}
1
.\]
Making the substitution 
$u=ds$ and $v=dt$, and arguing as before we find that 
\[
N(x)=
\sum_{\substack{m \in \N \\ \gcd(m,eq)=1
\\
m\leq \frac{\delta x}{|\b{v}(e)|}
}}
\mu(m)
\sum_{\substack{d\mid e\\ 
d\leq \frac{\delta x}{|\v(e/d)|m}
}}\mu(d)
\sum_{\substack{(s,t) \in \frac{x}{dm}\c{D}\cap \Lambda(e/d)
\\
(s,t)\equiv \overline{dm} (s_0,t_0)\md{q}}}
1.
\]

Now let $n \in \Z$ be such that $n\equiv \overline{dm} \md{q}$.
Then we can make the change of variables
$
(s,t)=n(s_0,t_0)+q(s',t')
$
in the inner sum. 
Noting that $\Lambda(e/d)$ defines a lattice in $\Z^2$
of determinant $e/d$, the inner sum is found to be 
\[
\frac{\mathrm{vol}(\c{D})x^2}{dem^2q^2}
+O\l(1+\frac{\frac{x}{dm}\partial{\c{D}}}{q|\b{v}(e/d)|}\r)
=\frac{\mathrm{vol}(\c{D})x^2}{dem^2q^2}
+O\l(\beta
\frac{x}{md|\b{v}(e/d)|} 
\r),
\]
with an absolute implied constant,
since the upper bound on $d$ implies that 
\[
1
\leq  \frac{\delta x}{dm |\b{v}(e/d)|} .
\] 
In summary, we have shown that 
$$
N(x)=
\sum_{\substack{m \in \N \\ \gcd(m,eq)=1
\\
m\leq \frac{\delta x}{|\b{v}(e)|}
}}
\mu(m)
\sum_{\substack{d\mid e\\ 
d\leq \frac{\delta x}{|\v(e/d)|m}
}}\mu(d)
\left(
\frac{\mathrm{vol}(\c{D})x^2}{dem^2q^2}
+O\l(\beta
\frac{x}{md|\b{v}(e/d)|} 
\r)\right).
$$
The contribution from the  error term is 
\[
\ll\beta x
\sum_{d|e}\frac{1}{d|\b{v}(e/d)|}
\sum_{m \leq \frac{\delta x}{d|\b{v}(e/d)|}}\frac{1}{m}
\ll
\beta x
\sum_{d|e}\frac{1}{d|\b{v}(e/d)|}
\log \left(2+\frac{\delta x}{d|\b{v}(e/d)|}\right)
.\]
The main term equals
\[
\frac{\mathrm{vol}(\c{D})x^2}{eq^2}
\sum_{\substack{m \in \N \\ \gcd(m,eq)=1
}}
\hspace{-0,5cm}
\frac{\mu(m)}{m^2}
\sum_{\substack{d|e\\
d\leq \frac{\delta x}{|\b{v}(e/d)|m}
}}\frac{\mu(d)}{d},
\]
since ~\eqref{eq:painkiller} implies that the extra constraint in $m$-sum is implied by the constraint in the $d$-sum.
But this is equal to 
\[
\frac{\mathrm{vol}(\c{D})x^2}{eq^2}
\sum_{d\mid e}\frac{\mu(d)}{d}
\sum_{\substack{m \in \N \\ \gcd(m,eq)=1
}}
\hspace{-0,5cm}
\frac{\mu(m)}{m^2}
+O
\l(
\frac{\mathrm{vol}(\c{D})x}{\delta q^2}
\cdot
\frac{1}{e}
\sum_{d\mid e}
|\b{v}(e/d)|
\r)
,\]
 which thereby completes the proof.  
\end{proof}

\subsection{Twisted Hooley $\Delta$-function over number fields}
\label{s:twist}

Adopting the notation of \S \ref{s:intro}, it
is now time to reveal the version of the {\em Hooley $\Delta$-function} 
that arises in our work. Let $K/\QQ$ be a number field
and let $\psi_K$ be a 
quadratic 
Dirichlet character on $K$. 
We let 
$\Delta:\mathcal{I}_{K} \to \R_{>0}$ be the function given by 
\begin{equation}
\label{eq:manannan}
\Delta(\fa;\psi_K)=
\sup_{\substack{u \in \R\\ 0\leq v\leq 1}}
\Big|
\sum_{\substack{
\fd\mid \fa\\
\eul^u 
<
\n_K\fd
\leq \eul^{u+v} 
}}
\psi_K(\mathfrak{d})
\Big|
,\end{equation}
for any integral ideal $\fa\in \mathcal{I}_K$.
We shall put 
$\Delta(\fa)=\Delta(\fa;\mathbf{1})$ for the corresponding function in which 
$\psi_K$ is replaced by the constant function $\mathbf{1}$.

We begin by  showing that $\Delta$ belongs to the class $\mathcal{M}_K$ of pseudomultiplicative functions introduced in \S \ref{s:Mk}. 
For coprime ideals 
$\fa_1,\fa_2 \subset \fo_K$, any ideal divisor  $\fd\mid \fa_1 \fa_2$ can be written uniquely
as 
$\fd=\fd_1 \fd_2$,  where $\fd_i\mid \fa_i$.
Therefore 
\[
\sum_{\substack{
\fd\mid \fa_1 \fa_2\\
\eul^u 
<
\n_K\fd
\leq \eul^{u+v} 
}}
\psi_K(\mathfrak{d})
=
\sum_{\substack{
\fd_1\mid \fa_1\\
}}
\psi_K(\mathfrak{d_1}) 
\sum_{\substack{
\fd_2\mid \fa_2 \\
\eul^{u-\log \n_K\fd_1} 
<
\n_K
\fd_2
\leq 
\eul^{u-\log \n_K\fd_1} 
\eul^{v} 
}}
\hspace{-0,5cm}\psi_K(\mathfrak{d}_2). 
\]
Thus  the triangle inequality yields
$\Delta(\fa_1\fa_2;\psi_K)
\leq 
\tau_K(\fa_1)
\Delta(\fa_2;\psi_K)$,
where $\tau_K$ is the divisor function on ideals of $\fo_K.$ 
This shows that 
$\Delta(\cdot,\psi_K)$
belongs to $\mathcal{M}_K$ 
and an identical argument confirms this for $\Delta(\cdot)$.

We shall need the following result proved in~\cite{sofos4hooley}.

\begin{lemma}
\label{lem:sofos}
Define the function 
\[
\widehat{\epsilon}(x)
=\sqrt{\frac{\log \log \log (16+x)}{\log \log (3+x)}},
\] 
for any $x\geq 1$
and recall the definition 
\eqref{eq:span-0} of  $\cP_K^\circ$.

 \begin{itemize}
\item[(i)] There exists a positive constant $c=c(K)$   
such that 
$$
\sum_{\substack{
\fa\in \cP_K^\circ  \text{ square-free} \\  
\n_K \fa  \leq x
}}\frac{\Delta(\fa)}{\n_K\fa}\ll (\log x)^{1+c\widehat{\epsilon}(x)}
.$$
\item[(ii)]   Let $\psi_K$ be a 
quadratic 
Dirichlet character on $K$ and let $W\in \NN$.
There exists a positive constant $c=c(K,\psi_K)$  
such that 
$$
\sum_{\substack{
\fa\in \cP_K^\circ  \text{ square-free} \\
\gcd(\n_K \fa,W)=1 \\
\n_K \fa  \leq x
}}\frac{\Delta(\fa;\psi_K)^2}{\n_K\fa}\ll  
(\log x)^{1+c\widehat{\epsilon}(x)}.
$$
\end{itemize}
The implied constant in both estimates is allowed to depend on $K$
and, in the second estimate, also on $W$ and the character $\psi_K$. 
\end{lemma}

\section{The lower bound}
\label{s:lower}

In order to prove the lower bound in Theorem \ref{th:upper}, 
we first  appeal to work of Frei, Loughran
and Sofos \cite{FS}.   
It follows from \cite[Thm.~1.2]{FS} that the desired lower bound holds when $\rho\geq 4$. Suppose that $\rho=3$.
Then \eqref{eq:rank} implies that 
in the fibration $\pi:X\to \PP^1$ there is at least one closed point $P\in \PP^1$ above which the singular fibre $X_P$
is split. 
Since the sum $c(\pi)$  defining the {\em complexity} of $\pi$ in \cite[Def.~1.5]{FS} is at most $4$ for  conic bundle  quartic del Pezzo surfaces, we infer that $c(\pi)\leq 3$ when $\rho=3$,
so that the lower bound in Theorem~\ref{th:upper} is a consequence of \cite[Thm.~1.7]{FS}.
Throughout this section, it therefore suffices to assume that $\rho=2$ and $\delta_0=0$, so that $X$ is a 
\textit{minimal} conic bundle surface.

Invoking  \cite[Thm.~1.6]{FS},  
 the lower bound in Theorem~\ref{th:upper} is a direct consequence of  the {\em divisor sum conjecture} 
that is recorded in \cite[Con.~1]{hypergods},
 for the relevant  data associated to the fibration $\pi$.   
Note that the principal result in~\cite{hypergods} 
only covers
cubic divisor sums, 
since we still lack  the technology to asymptotically evaluate 
divisor sums of higher degree with a power saving in the error term. 
The goal of this section is to
estimate certain quartic divisor sums,  
with a 
logarithmic saving in the error term, which  turns out to be sufficient for proving the lower bound in Theorem \ref{th:upper}.
The divisor sums relevant here shall involve
complicated quadratic symbols whose modulus tends to infinity,
a delicate task that will be the entire focus of this section. 

We proceed to explain the particular case of the  divisor sum conjecture that is germane here.
Assume that we are given
homogeneous polynomials 
$F_1,\ldots,F_n,G_1,\ldots,G_n \in \Z[x,y]$
with 
\[
F_i \text{ irreducible}, \quad
F_i\nmid G_i, \quad 2\mid \deg(G_i), \quad \text{and   $\quad 
\prod_{i=1}^n F_i$ separable}
.\] 
For each $i$ such that $F_i(1,0)\neq 0$, 
we define the associated  binary form
$\tilde F_i(x,y)=b_i^{d_i-1}F_i(b_i^{-1}x,y)$, as
in \eqref{eq:burger}, where $d_i=\deg F_i$ and $b_i=F_i(1,0)$. 
For such $i$ we let 
$\theta_i \in \bar\Q$ 
be a fixed root of $\tilde F_i(x,1)=0$. If, on the other hand,  
$F_i(x,y)$ is proportional to $y$, we  define $\theta_i=-F_i(0,1)$.   
We may assume that 
\begin{equation}\label{eq:egg}
\sum_{i=1}^n d_i=4
\end{equation}
and that $G_i(\theta_i,1) \notin \Q(\theta_i)^2$
for every $i$,
because 
in the correspondence outlined in \cite{FS}, 
the binary forms $F_1,\dots,F_n$ are equal to the closed points  $\Delta_1,\dots,\Delta_n$ from \S \ref{s:intro}.  Indeed, 
 under  this correspondence,
the statement 
$G_i(\theta_i,1) \notin \Q(\theta_i)^2$ is equivalent to the singular fibre above  $\Delta_i$ being non-split, which holds for any $i$ since we are  working with minimal conic bundle surfaces.

Let  
\begin{equation}\label{eq:i'm on a train}
f(d)=\prod_{p\mid d}\l(1-\frac{2}{p}\r).
\end{equation}
We need to prove that
there exists a finite set of primes $S_{\text{bad}}=
S_{\text{bad}}(F_i,G_i)$ 
such that for all $W\in \N$,  all
$(s_0,t_0) \in \Zp^2$, and all non-empty compact  discs 
$\c{D} \subset \R^2$, which together satisfy the conditions 
\begin{itemize}
\item[(C1)]
$p\in S_{\text{bad}}
\Rightarrow 
p\mid W$; 
\item[(C2)] $\prod_{i=1}^n F_i(s_0,t_0)\neq 0$;
\item[(C3)] 
$(s,t)\in \R^2\cap \c{D}\Rightarrow \prod_{i=1}^n F_i(s,t)\neq 0$; and
\item[(C4)]
for all $(s,t) \in \Zp^2\cap x\c{D}$ with  $x\geq 1$ and $(s,t)\equiv (s_0,t_0)\md{W}$ we have 
$$
\l(\frac{G_i(s,t)}{F_i(s,t)_{W}}\r)
=1;
$$
\end{itemize}
we have the lower bound  $D_{W}(x)\gg x^2$, where
\begin{equation}\label{eq:low}
D_W(x)=
\sum_{\substack{(s,t) \in \ZZp^2\cap x\c{D}\\ (s,t)\equiv (s_0,t_0)\bmod{W}}}
\prod_{i=1}^n
\left(
f(F_i(s,t)_{W})
\sum_{d\mid F_i(s,t)_{W}}
\left(
\frac{G_i(s,t)}{d}
\right)
\right).
\end{equation}
Here, we  recall the notation
$m_W=\prod_{p\nmid W} p^{\nu_p(m)}$ for all $m,W \in \N$.

We shall prove this conjectured lower bound when 
$S_{\text{bad}}$
is taken to be the set of all primes up to a constant $w=w(F_i,G_i)$.
In what follows  we shall often write that 
we  need to enlarge $w$.  This statement is to be  interpreted as having already taken a very large constant $w$ at the outset of  the proof of the conjecture, rather than  increasing $w$ within the confines of the lower bound arguments.
The primary goal of this section is now to establish the following bound, 
which directly leads to the lower bound  in Theorem \ref{th:upper}.

\begin{proposition}
\label{lem:low}
Let $F_i,G_i,f$ be as above. Then there exists a constant $w=w(F_i,G_i)$ such for any $W,(s_0,t_0),\c{D}$ 
satisfying (C1)--(C4)
as above, we have
$$D_W(x)\gg x^2.$$ 
Here
the implied constant depends on
$F_i,G_i,s_0,t_0,\c{D},w$ and $W$, but not on $x$.
\end{proposition}

Suppose that  
$\nu>\nu_p({W})$ for all $p\mid W$ and write
$W_0=\prod_{p\mid W } p^\nu$.
Then, since every summand in~\eqref{eq:low} is non-negative and
$F_i(s,t)_{W}=F_i(s,t)_{W_0}$ for all $1\leq i\leq n$, 
we conclude that 
$D_W(x)\geq D_{W_0}(x)$.
In this way
we see that it will suffice to prove the lower bound 
in Proposition  \ref{lem:low} under the assumption that 
$W=\prod_{p\mid W } p^\nu$ with
\[
\nu>\max_{\substack{1\leq i \leq n\\  p\mid W}}\{\nu_p(F_i(s_0,t_0))\}.
\] 
In this case the identity
$
F_i(s_0+p^\nu X,t_0+p^\nu Y)
\equiv 
F_i(s_0,t_0)
\md{p^\nu}
$
guarantees that  
$\nu_p(F_i(s,t))=\nu_p(F_i(s_0,t_0))$
for 
any $(s,t)$
 appearing in the outer summation of \eqref{eq:low} and any $p\mid W$.
Hence, for such $(s,t)$, we can always assume that 
\begin{equation}
\label{eq:lomatth}
F_i(s,t)_W=
|F_i(s,t)|
\prod_{p\mid W}p^{-\nu_p(F_i(s_0,t_0))}
.\end{equation}

\subsection{Dirichlet's hyperbola trick}
\label{s:dirihyptrick}
Let $i\in \{1,\dots,n\}$.
For any  $(s,t) \in \Z^2$ appearing in~\eqref{eq:low},
let
\[
r_i(s,t)
=
\sum_{k\mid  F_i(s,t)_W}
\left(
\frac{G_i(s,t)}{k}
\right).
\]
Then, possibly on enlarging $w$,  
it follows from Lemma \ref{lem:reuss} that
$$
r_i(s,t)
=\sum_{\substack{
\mathfrak{d}
\mid (b_i s-\theta_i t)
\\
\gcd(\n_i \fd,W)=1
\\
\fd\in \cP_i
}
}
\psi_i(\fd)
,
$$
where $\fd$ runs over   integral ideals of $K_i=\Q(\theta_i)$,
$\n_i$ denotes the ideal norm $\n_{K_i/\Q}$  
and $\cP_i=\cP_{K_i}$, in the notation of \eqref{eq:span}.
Furthermore, for all 
$(s,t)$ in~\eqref{eq:low},
we have
\[
\n_i \fd\leq \n_i(b_is-\theta_i t)
=|\tilde F_i (b_is,t)|
\leq c_i x^{d_i}
,\]
for some positive constant $c_i$ that depends at most on $F_i$ and $\c{D}$.
We define  
\[X=x \max\{ 
c_1^{\frac{1}{d_1}}, \dots, c_n^{\frac{1}{d_n}}\},
\]
so that the previous inequality becomes
$\n_i \fd \leq X^{d_i}$.

On relabelling the indices we may suppose that  $d_n=\min_{1\leq i\leq n} d_i$.
In particular, we have
\begin{equation}
\label{eq:relabel}
d_n
\leq  \min_{1\leq i\leq n}\deg(\Delta_i)
.\end{equation} 
Suppose that $n>1$. Then  for each $i\in \{1,\dots,n-1\}$
and $(s,t)$ appearing in~\eqref{eq:low}, we set
\begin{align*}
r_i^{(0)}(s,t)
&=\sum_{\substack{
\mathfrak{d}
\mid (b_i s-\theta_i t),~
\fd\in \cP_i\\ 
\\
\gcd(\n_i \fd,W)=1
\\
\n_i \fd\leq X^{\frac{d_i}{2}}
}
}
\psi_i(\fd),\quad 
r_i^{(1)}(s,t)
=
\sum_{\substack{
\mathfrak{e}
\mid (b_i s-\theta_i t), ~\fe\in \cP_i\\ 
\\
\gcd(\n_i \fe,W)=1
\\
\n_i \fe\leq X^{-\frac{d_i}{2}} F_i(s,t)_W
}
}
\psi_i(\fe).
\end{align*}
Dirichlet's hyperbola trick implies that 
\begin{equation}
\label{eq:dec}
r_i(s,t)
= 
r_i^{(0)}(s,t)+r_i^{(1)}(s,t).
\end{equation}
Indeed, if  $(b_is-\theta_i t)_W$ denotes the part of the ideal 
$(b_is-\theta_i t)$ that is composed solely of prime ideals whose norms are coprime to $W$, as in  \eqref{eq:aW},   then the sum in $r_i(s,t)$ is over ideals $\fd, \fe$ such that $\fd\fe=(b_is-\theta_i t)_W$.
 Recalling (C4),   
 it follows from part (ii) of Lemma  \ref{lem:reuss}  
that   $\psi_i((b_is-\theta_i t)_W)=1$. This  concludes the proof of \eqref{eq:dec}.

We proceed by  introducing the quantity
\begin{equation}
\label{eq:defL}
L=(\log x)^\alpha
,\end{equation}
for some $\alpha>0$ that will be determined in due course.
(When $n>1$ we shall take $\alpha$ to be a large constant, but when $n=1$ it will be important to restrict to $0<\alpha<1$.)
For $(s,t)$ appearing in~\eqref{eq:low},
we proceed by defining
\begin{align*}
r_n^{(0)}(s,t)
&=
\sum_{\substack{
\mathfrak{d}\mid (b_ns-\theta_n t), ~\fd\in \cP_n
\\
\gcd(\n_n \fd,W)=1
\\
\n_n\fd\leq L^{-1}
{X^{\frac{d_n}{2}}}
}
}
\psi_n(\mathfrak{d}),\qquad
r_n^{(1)}(s,t)=
\sum_{\substack{
\mathfrak{e}
\mid (b_ns-\theta_n t), ~\fe\in \cP_n
\\
\gcd(\n_n\fe,W)=1
\\
\n_n\fe\leq L^{-1}X^{-\frac{d_n}{2}} F_n(s,t)_W}}
\psi_n(\mathfrak{e})
\end{align*}
and 
\[
r_n^{(\infty)}(s,t)=
\sum_{\substack{
\mathfrak{d}\mid (b_ns-\theta_n t), ~\fd\in \cP_n
\\
\gcd(\n_n\fd,W)=1
\\
L^{-1}X^{\frac{d_n}{2}}
<
\n_n\fd
<
L X^{\frac{d_n}{2}}
}
}
\psi_n(\mathfrak{d})
.\]
As before, we may now write
\begin{equation}
\label{eq:decn}
r_n(s,t)
= 
r_n^{(\infty)}(s,t)
+
r_n^{(0)}(s,t)+r_n^{(1)}(s,t)
.
\end{equation}
For each
$\b{j}=(j_1,\ldots,j_n) \in \{0,1\}^n$,
we define 
\[
D_{\b{j}}(x)
=
\sum_{\substack{(s,t) \in \ZZp^2\cap x\c{D}\\ (s,t)\equiv (s_0,t_0)\bmod{W}}}
\prod_{i=1}^n
f(F_i(s,t)_W)
r_i^{(j_i)}(s,t)
,\]
and
$$
D_\infty(x) 
=
\sum_{\substack{(s,t) \in \ZZp^2\cap x\c{D}\\ (s,t)\equiv (s_0,t_0)\bmod{W}}}
r_n^{(\infty)}(s,t)
\prod_{i=1}^{n-1}
r_i(s,t),
$$
in which  we recall the definition 
\eqref{eq:i'm on a train} of $f$.
(Here, we recall our convention that products over empty  sets are equal to $1$.) 
Injecting~\eqref{eq:dec} and~\eqref{eq:decn} into~\eqref{eq:low} yields
\[
D_W(x)-
\sum_{\b{j} \in \{0,1\}^n} 
D_{\b{j}}(x)
\ll 
D_{\infty}(x)
.\]
The validity of Proposition~\ref{lem:low} is therefore assured, provided we can show that 
\begin{equation}
\label{eq:daniel}
D_{\b{j}}(x)
\gg x^2
\end{equation}
and
\begin{equation}
\label{eq:hool}
D_{\infty}(x)
=o(x^2).
\end{equation}
We shall devote \S\S\ref{ss:thin}--\ref{ss:nairhool} to the proof of \eqref{eq:hool} and \S \ref{ss:small} to the proof of \eqref{eq:daniel}.

\subsection{The generalised Hooley $\Delta$-function}
\label{ss:thin} 

In this subsection we initiate  the proof of 
 \eqref{eq:hool}.
Define 
\begin{equation}\label{eq:Ainfty}
A_n^{(\infty)}(x)
=
\left\{(s,t) \in \Zp^2\cap x\c{D}:
\begin{array}{l}
(s,t)\equiv (s_0,t_0)\bmod{W}
\\
\exists \fd \in \cP_n  \text{ such that:}\\
~\bullet~ \fd\mid (b_ns-\theta_n t)_W\\ 
~\bullet~L^{-1}X^{\frac{d_n}{2}}
<
\n_n\fd
< 
L
X^{\frac{d_n}{2}} 
\end{array}
\right\}.
\end{equation}
It  immediately follows that 
\[
D_{\infty}(x)
=
\sum_{(s,t) \in A_n^{(\infty)}(x)}
r_n^{(\infty)}(s,t)
\prod_{i=1}^{n-1}
r_i(s,t).
\]
Defining 
\begin{equation}
\label{def:Binfty}
B_{\infty}(x)
=
\sum_{(s,t) \in A_n^{(\infty)}(x)}
\prod_{i=1}^{n-1}
r_i(s,t)
,\end{equation}
we use Cauchy's inequality to arrive at
\[
D_{\infty}(x)
\leq 
B_{\infty}(x)^{\frac{1}{2}}
\l(
\sum_{(s,t) \in A_n^{(\infty)}(x)}
\Big|r_n^{(\infty)}(s,t)\Big|^2
\prod_{i=1}^{n-1}
r_i(s,t)
\r)^{\frac{1}{2}}
.\]

Recall the definition \eqref{eq:manannan} of the twisted Hooley $\Delta$-function
$\Delta(\fa;\psi_n)$ associated to the Dirichlet character $\psi_n$ and any integral ideal $\fa$.
Putting
\begin{equation}\label{eq:event_horizon}
H_{\infty}(x)
=
\sum_{\substack{(s,t) \in \ZZp\cap x\c{D}\\ (s,t)\equiv (s_0,t_0)\bmod{W}}}
\hspace{-0,3cm}
\Delta((b_ns-\theta_n t);\psi_n)_W^2 
\prod_{i=1}^{n-1}
r_i(s,t),
\end{equation}
and partitioning the interval
$(L^{-1}X^{\frac{d_n}{2}}, 
L
X^{\frac{d_n}{2}})$
into at most  $O(\log \log x)$ 
$\eul$-adic 
intervals,
we deduce  that  
\[
\sum_{(s,t) \in A_n^{(\infty)}(x)}
\Big|r_n^{(\infty)}(s,t)\Big|^2
\prod_{i=1}^{n-1}
r_i(s,t)
\ll
(\log \log x)^2
H_{\infty}(x)
.\]
In summary, we have shown that
\[
D_{\infty}(x)
\ll
(\log \log x)\sqrt{
B_{\infty}(x)
H_{\infty}(x)}
.\]
Therefore, in order to prove~\eqref{eq:hool},
it will be sufficient to prove
that there exists a constant $\delta>0$,
that depends only on the data  given at the start of \S\ref{s:lower}, 
such that
\begin{equation}
\label{eq:hool1}
B_{\infty}(x)
\ll x^2
(\log x)^{-\delta}
\end{equation}
and
\begin{equation}
\label{eq:hool2}
H_{\infty}(x)
\ll
x^2 (\log x)^{o(1)}
.\end{equation}
We shall 
call
$B_{\infty}(x)$
the \textit{interval sum}
and 
$H_{\infty}(x)$ 
 the \textit{Bret\`{e}che--Tenenbaum sum}. 

\subsection{The interval sum}
By recycling work of la Bret\`{e}che and Tenenbaum~\cite[\S 7.4]{bret-ten}, the 
case $n=1$ is easy to handle. 
Indeed, in this case $F_1$ is an irreducible quartic form  
and  \eqref{def:Binfty}  becomes  
$$
B_{\infty}(x)=\# A_1^{(\infty)}(x)\leq 
\sharp
\left\{(s,t) \in \Zp^2\cap x\c{D}:
\begin{array}{l}
(s,t)\equiv (s_0,t_0)\bmod{W}
\\
\exists \fd \in \cP_1  \text{ such that:}\\
~\bullet~ \fd\mid (b_1s-\theta_1 t)_W\\  
~\bullet~ X^2/L  
<
\n_1\fd
< 
L
X^2
\end{array}
\right\}.
$$
Note that assumption~(C2) 
ensures that 
$|F_1(s,t)| \asymp 1$ whenever $(s,t) \in \c{D}$.
Increasing $w$ so that 
every prime factor of $b_1$ also divides $W$,
shows that
\[
\tilde F_1(b_1s,t)_W=(b_1^{d_1-1}F_1(s,t))_W=F_1(s,t)_W.
\]
Thus it follows from ~\eqref{eq:lomatth} 
that 
$\tilde F_1(s,t)_W \asymp |F_1(s,t)|$, 
for implied constants that depend on 
$F_1,s_0,t_0,w$ and $W$.
Hence 
$$
\n_1((b_1s-\theta_1 t)_W)=\tilde F_1(b_1s,t)_W 
\asymp 
|F_1(s,t)| \asymp x^4\asymp X^4.
$$ 
Therefore, on introducing $\fe$ through the factorisation
$\fd\fe=(b_1 s-\theta_1 t)_W$,
we can infer that we must have either
\[
X^2/L
\ll
\n_1 \fd
\ll 
X^2
\quad
\text{ or }
\quad
X^2/L
\ll
\n_1 \fe
\ll 
X^2
.\]
Without loss of generality we shall assume that we are in the former setting.
Therefore there exist constants $c_0,c_1>0$ such that
$$
B_\infty(x)
\ll
\sharp
\left\{(s,t) \in \Zp^2\cap x\c{D}:
\begin{array}{l}
(s,t)\equiv (s_0,t_0)\bmod{W}
\\
\exists d\mid F_1(s,t) \text{ s.t.\
$c_0x^2/L
<
d
< 
c_1
x^2$}
\end{array}
\right\}.
$$
But now we can employ the bound~\cite[Eq.~(7.41)]{bret-ten},
with
$$
T=F_1,\quad \Xi= \xi=x, \quad y_1=c_0x^2/L, \quad y_2=c_1x^2, \quad \text{ and } \quad 1\ll \sigma,\vartheta\ll 1.
$$
This implies that for any $\eta \in (0,\frac{1}{2})$, 
we have 
\[
B_\infty(x)\ll
x^2
\l(\frac{L}{(\log x)^{Q(2\eta)}}
+
\frac{\log \log x}{(\log x)^{Q(1+\eta)}}
\r)
,\]
where $Q(\lambda)=\lambda \log \lambda - \lambda +1$.
In particular, $Q(2\eta)\to 1$ as 
$\eta\to 0+$ and 
$Q(1+\eta)>0$ for all $\eta>0$.
Recalling the definition~\eqref{eq:defL} of $L$, this
means that provided $\alpha<1$, 
 we may choose $\eta>0$ small enough (but away from $0$), so as to ensure that \eqref{eq:hool1} holds when $F$ is irreducible. 

It remains to establish \eqref{eq:hool1} when 
$n>1$.  
In this case \eqref{eq:relabel} implies that $d_n=\deg(F_n)\leq 2$.
Fix $\eta \in (0,1)$.
To estimate $B_\infty(x)$,
drawing inspiration from~\cite[\S 9.3]{bret-ten},
 we shall divide
the terms in the sum~\eqref{def:Binfty}
into two categories.

\subsubsection*{First case:  $(b_n s-\theta_n t)$ has  many prime divisors}

We denote by 
$B_{\infty}^{(1)}(x)$ the contribution 
to 
$B_{\infty}(x)$ from $(s,t)$ for which 
$
\Omega_{n}((b_ns-\theta_n t)_W)>(1+\eta) \log \log x,
$ 
where $\Omega_n(\fa)=\Omega_{K_n}(\fa)$ is the total number of prime ideal 
factors of an ideal $\fa\subset \fo_{K_n}$. 
Recall that, as in~\S\ref{s:dirihyptrick},
we denote $\n_{K_n}(\fa)$
by $\n_{n}(\fa)$.
We have
\begin{equation}
\label{eq:sumst}
B_{\infty}^{(1)}(x)
\leq 
(\log x)^{-(1+\eta) \log(1+\eta)}
\hspace{-0,3cm}
\sum_{(s,t) \in \ZZp^2\cap x\c{D}}
(1+\eta)^{\Omega_n((b_n s-\theta_n t)_W)}
\prod_{i=1}^{n-1}
r_i(s,t),
\end{equation}
since $(1+\eta)^{-(1+\eta)\log \log x}=(\log x)^{-(1+\eta) \log(1+\eta)}$.
Our plan is now to apply Lemma~\ref{t:NT}
for $N=n$, with 
$f_N(\fa)=(1+\eta)^{
\Omega_n(\fa_W) 
}$ and 
$$
f_i(\fa)=\sum_{\substack{\fd\mid \fa\\ \fd\in \cP_i}}\psi_i(\fd), 
$$
for $i< N$.
Fix any
$\ve>0$.
It is easy to see that if $i<N$ then there exists 
$B>0$ 
such that $f_i\in \c{M}_{K_i}(2,B,\ve)$. Thus,  in the notation of Lemma~\ref{t:NT},  one can take 
\begin{equation}
\label{eq:iln}
i<N\Rightarrow 
\epsilon_i=\epsilon
.\end{equation}
When $i=N$, however,  we will show that for every $\ve>0$ there exists $w$ such that if $W$ is given by~\eqref{eq:WW}
then \[(1+\eta)^{\Omega_n(\fa_W)} \in \c{M}_{K_n}(1+\eta,1,\ve).\]
Indeed, we have 
\[(1+\eta)^{\Omega_n(\fa_W)}
=\prod_{\substack{\fp^\xi\| \fa \\\gcd(\n_n \fp,W)=1}}(1+\eta)^\xi
\leq 
\prod_{\substack{\fp^\xi\| \fa \\ \n_n \fp >w }}(1+\eta)^\xi .
\]
Taking $w\geq 2^{1/\ve}$, so that 
$(1+\eta)\leq w^\ve$,  yields
\[
\prod_{\substack{\fp^\xi\| \fa \\ \n_n \fp >w }}(1+\eta)^\xi 
\leq 
\prod_{\substack{\fp^\xi\| \fa \\ \n_n \fp >w }}w^{\ve\xi} 
\leq 
\prod_{\substack{\fp^\xi\| \fa \\ \n_n \fp >w }}(\n_n \fp)^{\ve\xi} 
\leq (\n_n \fa)^\ve
.\]
This means that
 in the notation of Lemma~\ref{t:NT} one can take 
\begin{equation}
\label{eq:ilndio} 
\epsilon_N=\epsilon
.\end{equation}
Furthermore, we shall take $G=\ZZ^2$  and $\c{R}=x\c{D}$.
Thus  $q_G=1$,
$\c{R}$ is regular and 
we have $V\asymp x^2$ and $K_{\c{R}}\asymp x \log x$, in the notation 
of the lemma.  
This means that for large $x$ we can take $c_1=1$, hence by~\eqref{eq:egg}, \eqref{eq:iln} and~\eqref{eq:ilndio} 
we have 
\[
\sum_{i=1}^N d_i \ve_i=4\ve.
\]
Therefore, assuming that $\ve\in (0,1)$ is fixed, 
the relevant constant  in Lemma~\ref{t:NT}
is 
$\ve_0=
\max\{5,20+12\ve\} 4\ve
\leq 199 \ve$.
This   shows that if  $\ve$ is fixed and 
$200 \ve <1/3$ then 
\[
\frac{K_{\c{R}}^{1+\ve_0+\ve}}{\lambda_G}
\ll (x \log x)^{1+200 \ve} \ll x^{3/2}
,\] hence the secondary term
of Lemma~\ref{t:NT}
makes a satisfactory contribution.
The contribution
of 
the first term of Lemma~\ref{t:NT}
towards the sum in~\eqref{eq:sumst}
is 
\begin{align*}
&\ll 
\frac{x^2}{(\log x)^n}
\exp\l(\sum_{i =1}^{n-1}
\sum_{\substack{\fp\in \cP_i^\circ \\ 
\n_i\fp \ll x^2}}\frac{1+\psi_i(\fp)}{\n_i \fp}
+
(1+\eta)
\sum_{\substack{\fp\in \cP_n^\circ \\ 
\n_n\fp\ll x^2}}\frac{1}{\n_n \fp}\r)\\
&\ll
\frac{x^2}{(\log x)^n}
\exp((n-1)\log \log x+(1+\eta) \log \log x)\\
&\ll x^2 (\log x)^{\eta}.
\end{align*}
The proof of these estimates is standard and will not be repeated here. (See Heilbronn \cite{H}, for example.)
Thus
$B_{\infty}^{(1)}(x)
\ll
x^2 (\log x)^{-(1+\eta)\log (1+\eta)+\eta}$.
The exponent of the logarithm is 
strictly negative for all $\eta>0$, which is clearly sufficient for \eqref{eq:hool1}.

\subsubsection*{Second case:
$(b_n s-\theta_n t)$ has  few prime divisors}
We denote by $B_\infty^{(2)}(x)$ the contribution to $B_{\infty}(x)$ from $(s,t)$ for which 
$
\Omega_n((b_n s-\theta_n t)_W) \leq (1+\eta) \log \log x.
$
Recall from the definition \eqref{eq:Ainfty} 
of $A_n^{(\infty)}(x)$ 
that there exists $\fd \in \cP_n$ such that 
$\fd \mid (b_n s-\theta_n t)$, with 
$\gcd(\n_n\fd,W)=1$ and 
\[
L^{-1}X^{\frac{d_n}{2}}
<
\n_n\fd
<
L
X^{\frac{d_n}{2}}
.\]
Condition (C3) ensures that 
$
\n_n((b_ns-\theta_n t)_W)\asymp X^{d_n}
$.
Defining $\fe$ via the factorisation
$\fd\fe=(b_n s-\theta_n t)_W$,
we can then infer 
that $\gcd(\n_n\fe,W)=1$ and $\fe\in \cP_n$, with
$
L^{-1}X^{\frac{d_n}{2}}
\ll
\n_n\fe
\ll
L
X^{\frac{d_n}{2}},
$
where the implied constants  depend at most on $\c{D}$ and $F_n$.
Note
that 
\[
\Omega_n(\fd)
+
\Omega_n(\fe)
=\Omega_n((b_n s-\theta_n t)_W)
\leq 
(1+\eta) \log \log x.
\]
Thus, either 
$\Omega_n(\fd)
\leq  
\frac{1}{2}
(1+\eta) \log \log x$,
or 
$\Omega_n(\fe)
\leq \frac{1}{2}
(1+\eta) \log \log x$.
We will assume without loss of generality that we are in  the latter case.

It follows that 
\[
B_{\infty}^{(2)}(x)
\ll
\sum_{\substack{
\fe\in \cP_n\\
L^{-1}X^{\frac{d_n}{2}}
\ll
\n_n\fe
\ll
L
X^{\frac{d_n}{2}}
\\ 
\Omega_n(\fe)
\leq \frac{1}{2}
(1+\eta) \log \log x
\\ 
\gcd(\n_n\fe,W)=1
}}
B_{\fe}(x),
\]
where
\[
B_{\fe}(x)
=
\sum_{\substack{
(s,t) \in \ZZp^2\cap x\c{D}
\\
(s,t)\equiv (s_0,t_0)\bmod{W}
\\
\fe | (b_ns -\theta_n t) 
}}
\prod_{i=1}^{n-1}
r_i(s,t).
\]
This is a non-archimedean version of Dirichlet's hyperbola trick, where instead of looking at the complimentary divisor 
to reduce the size, we have tried to reduce the number of prime divisors.
Lemma \ref{lem:honey} implies that  the condition
$\fe \mid (b_ns -\theta_n t)$
defines a lattice in $\Z^2$
of determinant $e=\n_n\fe$, which we shall call $G$. 
Hence we may write
$$
B_{\fe}(x) 
=
\sum_{\substack{
(s,t) \in \Zp^2\cap x\c{D} \cap G
\\
(s,t)\equiv (s_0,t_0)\bmod{W}
}}
\prod_{i=1}^{n-1}
r_i(s,t).
$$
Let  $\b{v}\in \Z^2$ be such that $|\b{v}|=\max\{|v_1|,|v_2|\}$ is the first successive minimum of  $G$. 
Lemma~\ref{t:NT}
can  be applied with $\c{R}=x\c{D}$,
$q_G=e$,
$N=n-1$,
and 
\[
f_i(\fa)=\sum_{\fd\mid \fa}\psi_i(\fd), 
\]
for $1\leq i\leq n-1$. 
For such $f_i$ one can take $\ve_i$ in Lemma~\ref{t:NT} to be arbitrarily small, whence 
\[
B_\fe(x)\ll
x^2\frac{h^*(e)}{e}
+\frac{x^{1+\ve}}{|\b{v}|},
\]
for any $\ve>0$, where 
$$
h^*(e)=\prod_{p\mid k} \left(1-\frac{\bar\rho_1(p)+\dots +\bar\rho_{n-1}(p)}{p+1}\right)^{-1}.
$$
(Note that $h_W^*(e)=h^*(e)$, since $\gcd(e,W)=1$.)

We have 
$e=\n_n\fe\ll LX^{\frac{d_n}{2}}$ and so
$|\b{v}|\ll \sqrt{LX^{\frac{d_n}{2}}}\leq \sqrt{LX}$, since $d_n\leq 2$. Since $F_n$ is irreducible, we note that 
 $d_n=1$ when  $F_n(\v)=0$.
Define  
$g(e)=\#\{\fe\in \cP_n: \n_n\fe=e\}.$ 
The second term is therefore seen to  make the overall contribution
\begin{align*}
&\ll
x^{1+\ve}
\sum_{\substack{
|\v|\ll \sqrt{LX} \\
F_n(\v)\neq 0}} \frac{1}{|\v|} \sum_{e\mid F_n(\v)} g(e)
+
x^{1+\ve}
\sum_{\substack{
|\v|\ll \sqrt{LX} \\
F_n(\v)= 0}} \frac{1}{|\v|} \sum_{e\ll L\sqrt{X}} g(e)
\ll 
x^{\frac{3}{2}+2\ve},
\end{align*}
which is satisfactory.

Next,  the  overall contribution from the term $x^2 h^*(e)/e$
is  $O( x^2 \Sigma)$, where
\[
\Sigma=
\sum_{\substack{
L^{-1}X^{\frac{d_n}{2}}
\ll
e
\ll
L
X^{\frac{d_n}{2}}
\\ 
\Omega(e)
\leq \frac{1}{2}
(1+\eta) \log \log x
\\ 
\gcd(e,W)=1}}
\frac{g(e)h^*(e)}{e}
.\]
Letting $A=
\l(\frac{1+\eta}{2}\r)^{-1}
>1$,
we get
\[
\Sigma\ll
(\log x)^{\frac{\log A}{A}}
\sum_{\substack{
L^{-1}X^{\frac{d_n}{2}}
\ll
e
\ll
L
X^{\frac{d_n}{2}}
\\ 
\gcd(e,W)=1}}
\frac{g(e)h^*(e)}{e} A^{-\Omega(e)}
.\]
Put 
$$
S(y)=
\sum_{\substack{
e
\leq y
\\ 
\gcd(e,W)=1}}
g(e)h^*(e)A^{-\Omega(e)}.
$$
Then it follows from Shiu's work \cite{shiu} 
that 
\begin{align*}
S(y)
\ll   \frac{y}{\log y}
\exp\l(A^{-1}\sum_{\substack{p\leq y\\ p\nmid W}}\frac{g(p)h^*(p)}{p}\r)
&\ll
 \frac{y}{\log y}
\exp\l(A^{-1}\sum_{\substack{p\leq y\\ p\nmid W}}\frac{\bar\rho_n(p)}{p}\r)\\
&\ll
y(\log y)^{\frac{1}{A}-1}.
\end{align*}
Partial summation now leads to the estimate
\begin{align*}
B_\infty^{(2)}(x)
&\ll
x^2
(\log \log x)
(\log x)^{\frac{\log A}{A}+\frac{1}{A}-1}\\
&=
x^2
(\log \log x)
(\log x)^{\frac{\eta-1}{2}-(\frac{1+\eta}{2})\log (\frac{1+\eta}{2})}.
\end{align*}
The exponent of $\log x$ 
is strictly negative for all $\eta\in (0,1)$, which thereby  completely settles the proof of~\eqref{eq:hool1}.

\subsection{The Bret\`{e}che--Tenenbaum sum}   
\label{ss:nairhool}
We saw in \S \ref{s:twist}
that the Hooley $\Delta$-function defined in \eqref{eq:manannan} belongs to $\mathcal{M}_n$.
The stage is now set for an application of Lemma~\ref{t:NT}
with
$N=n$ and $G=\ZZ^2$, and with
$f_N(\fa)=
\Delta(\fa;\psi_n)^2$ and 
$
f_i(\fa)=\sum_{\fd\mid \fa} \psi_i(\fd),
$
for $i<N$.
For such $f_i$ one can take $\ve_i$ in Lemma~\ref{t:NT} to be arbitrarily small, whence 
this  gives
$$  
H_{\infty}(x)
\ll
\frac{x^2}{\log x}
E_{\Delta(\cdot;\psi_n)^2}(x^2;W) 
$$
in \eqref{eq:event_horizon}.
The statement of  \eqref{eq:hool2} now follows from  part (ii) of 
Lemma \ref{lem:sofos}. 

\subsection{Small divisors}\label{ss:small}
In this subsection we establish 
 \eqref{eq:daniel}, as required to complete the proof of Proposition \ref{lem:low}.
When $n>1$, 
the proof 
 follows from the treatment in \cite{FS} and 
will not be repeated here.   
Thus, provided that one takes 
 $\alpha$ to be sufficiently large 
in the definition \eqref{eq:defL} of $L$, one gets an asymptotic formula for 
$D_{\b{j}}(x)$  with a  logarithmic  saving in the error term.
The proof of \eqref{eq:daniel} when $n=1$ is more complicated. 
In this case 
$F_1$ is an irreducible binary quartic form. In order
to simplify the notation, 
we shall drop the index $n=1$ in what follows 
(in particular, we shall denote $\cP_{K_1}=\cP_1$ by $\cP$).
Our task is to estimate
\[
D_j(x)
=
\sum_{\substack{(s,t) \in \ZZp^2\cap x\c{D}\\ (s,t)\equiv (s_0,t_0)\bmod{W}}}
f(F(s,t)_W)
r^{(j)}(s,t)
,\]
for $j\in\{0,1\}$. 
Opening up the definition of  $f(F(s,t)_W)$, 
it follows from parts (i) and (ii) of Lemma \ref{lem:reuss} that 
$$
f(F(s,t)_W)=\sum_{\substack{e\mid F(s,t)\\ \gcd(e,W)=1}} \frac{\tau(e)\mu(e)}{e}
=\sum_{\substack{\fe \mid (bs-\theta t)\\ \gcd(\n\fe,W)=1
\\ \fe\in \cP}}
 \frac{\tau(\fe)\mu(\fe)}{\n\fe},
$$
since $\tau(\n\fe)=\tau_{K_1}(\fe)=\tau(\fe)$, say,  for any  
$\fe\in \mathcal{P}$.

Let $y>0$.
The overall contribution to $D_j(x)$ from $\fe$ such that $\n\fe>y$ is 
$$
\ll 
\sum_{\substack{y<\n \fe\ll x^4\\ \gcd(\n\fe,W)=1
\\ \fe\in \cP}}
 \frac{\tau(\fe)|\mu(\fe)|}{\n\fe}
\sum_{\substack{(s,t) \in \ZZp^2\cap x\c{D}\\ 
\fe \mid (bs-\theta t)}}
r^{(j)}(s,t).
$$
The condition $\fe \mid (bs-\theta t)$ defines a lattice in $\ZZ^2$ of determinant $\n\fe$ by Lemma~\ref{lem:honey}. Thus  we can apply Lemma~\ref{t:NT}, finding that
$$
\sum_{\substack{(s,t) \in \ZZp^2\cap x\c{D}\\ 
\fe \mid (bs-\theta t)}}
r^{(j)}(s,t)\ll 
x^2\frac{h_W^*(\n\fe)}{\n\fe}+x^{1+\frac{\ve}{2}},
$$
for any $\ve>0$, where $h^*$ is given by \eqref{eq:faggot} with $N=1$.
Hence we arrive at the overall contribution 
\begin{align*}
&\ll 
x^2\sum_{\substack{\n \fe>y}}
(\n\fe)^{-2+\ve}
+
x^{1+\frac{\ve}{2}}\sum_{\substack{\n \fe\ll x^4}} (\n\fe)^{-1+\frac{\ve}{8}}
 \ll 
\frac{x^{2}}{\sqrt{y}}
+
x^{1+\ve},
\end{align*}
from $\n\fe>y$. Taking $y=\log\log x$, we therefore conclude that 
$$
D_j(x)=
\sum_{\substack{\n\fe\leq \log\log x\\ \gcd(\n\fe,W)=1
\\ \fe\in \cP}}
 \frac{\tau(\fe)\mu(\fe)}{\n\fe}
\sum_{\substack{(s,t) \in \ZZp^2\cap x\c{D}\\ 
\fe \mid (bs-\theta t)\\
(s,t)\equiv (s_0,t_0)\bmod{W}}}
r^{(j)}(s,t) +O\left(\frac{x^2}{\sqrt{\log\log x}}\right).
$$
Note that by enlarging $w$ we may assume that any  
prime factor of $b$ is present in the factorisation of $W$.

We henceforth  focus on the case $j=0$, the case $j=1$ being similar. 
First, we define for any 
$\fa\in \c{P}$ with $\gcd(\n\fa,W)=1$ the set 
\[
\c{H}(\fa)
=\big\{(s,t)\in \Z^2:\fa \mid (bs-\theta t)\big\}.
\]
By Lemma \ref{lem:honey}
there exists  $k=k(\fa)\in \Z$ such that
a vector $(s,t) \in \Z^2$
belongs to $\c{H}(\fa)$
if and only if 
$\n \fa \mid bs-k t$.  
Therefore, $\c{H}(\fa)$ is a lattice in $\Z^2$ of determinant 
$\n \fa$.  
Recalling the definition of $r^{(0)}(s,t)$ we obtain
\begin{equation}\label{eq:mould}
\begin{split}
D_{0}(x)
=~&
\sum_{\substack{\n\fe\leq \log\log x\\ \gcd(\n\fe,W)=1
\\ \fe\in \cP}}
 \frac{\tau(\fe)\mu(\fe)}{\n\fe}
\sum_{\substack{
\n\fd\leq L^{-1}X^2
\\
\gcd(\n \fd,W)=1\\
\fd\in \cP
}
}
\psi (\mathfrak{d})
\sum_{\substack{(s,t) \in \ZZp^2\cap x\c{D}\\ 
(s,t) \in \c{H}(\fd) \cap \c{H}(\fe)\\
(s,t)\equiv (s_0,t_0)\bmod{W}}}1\\
&+
O\left(\frac{x^2}{\sqrt{\log\log x}}\right).
\end{split}
\end{equation} 
In fact, for coprime integers $s,t$, 
part (i) of Lemma \ref{lem:reuss} ensures that  we only have 
$(s,t) \in \c{H}(\fd) \cap \c{H}(\fe)$ if the least common multiple 
$[\fd,\fe]$ of $\fd$ and $\fe$ belongs to $ \cP$.
It now follows from Lemma \ref{lem:honey} that there exists 
$k=k(\fd,\fe)\in \ZZ$ 
such that 
$(s,t) \in \c{H}(\fd) \cap \c{H}(\fe)$ 
if and only if 
$bs\equiv kt \bmod{M}$,
 where $M=[\n\fd,\n\fe]$ is the least common multiple  of $\n\fd$ and $\n\fe$.
We let  $\b{v}(M)=\b{v}(M;\fd,\fe)$ denote a fixed non-zero vector $(s,t)\in \ZZ^2$ of minimal length such that $bs\equiv kt \bmod{M}$.

Note that   $\gcd(b,k,M)=1$, 
since $\gcd(M,W)=1$ and we chose $W$ in such a way that any prime factor of $b$ also divides $W$.
The inner sum over $s,t$ is now in a form that is suitable for  Lemma \ref{lem:historical}, with 
$c=k$,  $e=M$ and
$$
 1\ll \delta(\cD)\ll 1, \quad \beta, \gamma\ll 1.	
$$
Arguing as in \cite[\S\S 4.3--4.5]{FS}, 
once inserted into \eqref{eq:mould},
the contribution from the main term 
(denoted by $M_{\boldsymbol{\psi}}$ in \cite{FS})
in
Lemma \ref{lem:historical} is  $\gg x^2$.
 This is satisfactory for \eqref{eq:daniel}.
It remains to consider the effect of substituting the error term 
in Lemma \ref{lem:historical}.

Let 
\[
r^*(m)=\#\{\fa\in \cP^\circ: \n\fa=m,~ \gcd(\n\fa,W)=1\},  
\]
for any $m\in \NN$, where we recall that  $\cP^\circ$ is the multiplicative 
span of prime ideals with residue degree $1$.
This function is multiplicative and has constant  average order.
We claim that 
$
r^*(cd)\leq r^*(c)r^*(d)
$
for all $c,d \in \N$, which we shall keep in use throughout this subsection. 
It is enough to consider the case
$c=p^a$ and $d=p^b$ for a rational prime
$p\nmid W$ 
with $r^*(p)\neq 0$.
Letting $\fp_1,\ldots,\fp_{m+1}$ be  
all the degree $1$ prime ideals above $p$, we easily see that 
$r^*(p^k)={ k+ m \choose m }$.
We therefore have to verify that 
\[
{ a+b+m \choose m }
\leq
{ a+m \choose m }
{ b+m \choose m },
\] 
for all integers 
$a,b,m\geq 0$. This is obvious when $m=0$. When $m\geq 1$ the inequality is 
equivalent to 
\[1\leq \prod_{i=1}^m\frac{(a+i)(b+i)}{i(a+b+i)},\] 
the validity of which is clear.  
 
The error term in Lemma \ref{lem:historical} is composed of two parts. 
 According to  \eqref{eq:paingiver}, the second part contributes 
\begin{align*}
&\ll x
\sum_{\substack{
\n\fd\ll x^2/L\\
\n\fe\leq \log\log x}}
 \frac{\tau(\fe)|\mu(\fe)|}{\n\fe}\cdot \frac{1}{M}
\sum_{d\mid M} \sqrt{d},
\end{align*}
with $M=[\n\fd,\n\fe].$ Taking $M\geq \n\fd=q$, say, and 
$$
\sum_{d\mid M} \sqrt{d}\leq \tau(\n\fe)\sqrt{\n\fe}
\sum_{d\mid q} \sqrt{d},
$$
we conclude that
the second part contributes
\begin{align*}
\ll x
\sum_{\substack{
q\ll x^2/L\\
\n\fe\leq \log\log x
}}
 \frac{\tau(\fe)^2|\mu(\fe)|}{\sqrt{\n\fe}}\frac{r^*(q)}{q}
\sum_{d\mid q} \sqrt{d}
&\ll x \log\log x
\sum_{\substack{q\ll x^2/L}}
 \frac{r^*(q)}{q}
\sum_{d\mid q} \sqrt{d}\\
&\ll x \log\log x
\sum_{\substack{cd\ll x^2/L}}
 \frac{r^*(c)r^*(d)}{c\sqrt{d}}\\
 &\ll x \log\log x
\sum_{\substack{c\ll x^2/L}}
 \frac{r^*(c)}{c}\sqrt{\frac{x^2}{cL}}\\
 &\ll L^{-\frac{1}{2}}x^2 \log\log x,
\end{align*}
in \eqref{eq:mould}. This is satisfactory for any $\alpha>0$ in \eqref{eq:defL}.

Finally, the overall contribution from the first part of the error term of
Lemma \ref{lem:historical} is
\begin{align*}
&\ll x
\sum_{\substack{
\n\fd\ll x^2/L\\
\n\fe\leq \log\log x\\
[\fd,\fe]\in \cP\\
\gcd(\n\fd\n\fe,W)=1
}}
 \frac{\tau(\fe)|\mu(\fe)|}{\n\fe}
\sum_{\substack{u \in \NN \\ u\mid M}}
\frac{1}{u|\b{v}(M/u)|} 
\log \left(2+\frac{ x}{u|\b{v}(M/u)|}\right).
\end{align*}
Here we recall that $\v(M/u)$ is a vector $(s,t)\in \ZZ^2$ of minimal length 
for which $bs\equiv kt \bmod{M/u}$.
In particular it also depends on $\fd$ and $\fe$ since $k$ does.
Put  $d=\n\fd$ and $e=\n\fe$, so that $M=[d,e].$  
If  $u\mid [d,e]$ then we 
claim that there is a factorisation $u=u'u''$ such that 
$u'\mid d$, 
$u''\mid e$ and such that $d/u'$ divides
$ [d,e]/u$.
To see this let $\nu_p(d)=\delta$ and $\nu_p(e)=\epsilon$ for any prime $p$. 
If $u\mid [d,e]$ then $\nu_p(u)\leq \max\{\delta,\epsilon\}$ for any prime $p$.
We take
$$
u'=\prod_{p^\nu \| u} p^{\min\{\nu,\delta\}} 
\quad \text{ and } \quad 
u''=\prod_{p^\nu \| u} p^{\nu-\min\{\nu,\delta\}}.
$$
It is clear that $u'\mid d$ and $u''\mid e$. Moreover, one easily checks that 
$$
\nu_p(d/u')=\delta-\min\{\nu,\delta\} 
\leq 
\max\{\delta,\epsilon\}-\nu=
\nu_p([d,e]/u),
$$
for any prime $p$, whence
 $d/u'\mid  [d,e]/u$.
In particular, this implies that $$
|\v([d,e]/u;\fd,\fe)|\geq |\v(d/u';\fd,\fe)|.
$$

Our argument so far shows that the term in which we are interested is
\begin{equation}\label{eq:gwr}
\ll x
\sum_{\substack{
e\leq \log\log x}}
\frac{\tau(e)^2}{e}
\sum_{
\substack{\fe\in \cP\\ 
\n\fe=e
}}
S(\fe),
\end{equation}
where
\begin{align*}
S(\fe)
&=
\sum_{\substack{d\ll x^2/L\\
\gcd(d,W)=1}}
\sum_{
\substack{\fd\in \cP\\ 
\n\fd=d
}}
\sum_{\substack{u'\mid d}}
\frac{1}{u'|\b{v}(d/u')|} 
\log \left(2+\frac{ x}{u'|\b{v}(d/u')|}\right)\\
&\leq
\sum_{\substack{u'\ll x^2/L\\ \gcd(u',W)=1}}
\frac{1}{u'}
\sum_{\substack{d'\ll x^2/(u'L) \\ \gcd(d',W)=1}}
\sum_{
\substack{\fd\in \cP\\ 
\n\fd=d'u'
}}
\frac{1}{|\b{v}(d')|} 
\log \left(2+\frac{ x}{|\b{v}(d')|}\right),
\end{align*}
with the caveat that 
$\v(d')$ still depends on $\fd$ and $\fe$. 
Moreover if there exists $\fd\in \cP$ with $\gcd(\n\fd,W)=1$ such that $\n\fd=d'u'$ then there exists 
 $\fd'\in \cP$ with $\gcd(\n\fd',W)=1$ such that $\n\fd'=d'$.
Hence $\fd '$ must divide
$(\b{v}(d')_1-\theta \b{v}(d')_2)$ and so it follows that 
$d' \mid F(\b{v}(d'))$.  
Furthermore, we note  that $|\v(d')|\ll \sqrt{d'}\ll x/\sqrt{L}$ in our upper bound for $S(\fe)$.

The contribution from $d',\fd$ for which $|\b{v}(d')|\leq x/(\log x)^\Upsilon$ is  
seen to be
\begin{align*}
&\ll \log x
\sum_{\substack{u'\ll x^2/L\\ \gcd(u',W)=1}}
\frac{r^*(u')}{u'}
\sum_{\substack{
\v=(v_1,v_2)\in \ZZ^2\\
0<|\b{v}|\leq x/(\log x)^\Upsilon}} \frac{1}{|\v|}
\sum_{\substack{d'\mid F(\v)}} r^*(d')\ll
x(\log x)
^{-\Upsilon+10},
\end{align*}
by \cite{nair}.  
Here we have used the fact that $r^*(d')\leq \tau_4(d')$ and 
\begin{equation}\label{eq:gwr2}
\sum_{\substack{u'\leq U}}
\frac{r^*(u')}{u'} \leq
 \sum_{\substack{u'\leq  U}}\frac{r_K(u')}{u'} \ll \log U,
\end{equation} where $r_K$ are the coefficients in the associated Dedekind zeta function. 
Once inserted into \eqref{eq:gwr} this contributes
$$
\ll x^2(\log x)^{-\Upsilon+10} \sum_{e\leq \log\log x} \frac{\tau(e)^2r^*(e)}{e}\ll 
x^2(\log x)^{-\Upsilon+9},
$$
which is satisfactory, 
on taking $\Upsilon$  sufficiently large.

In the opposite case, we plainly have 
$
d'\gg |\b{v}(d')|^2\geq x^2/(\log x)^{2\Upsilon
},
$
whence 
$\log (2+x/|\b{v}(d')|)\ll_\Upsilon \log \log x$. 
Moreover, the inequalities
$d'\ll x^2/(u'L)$
and $d'\gg  x^2/(\log x)^{2\Upsilon}$ together
provide us with 
$u'\ll (\log x)^{2\Upsilon}$. Thus it remains to study the  contribution
\begin{align*}
&\ll_\Upsilon \log\log x
\sum_{\substack{u'\ll (\log x)^{2\Upsilon}}}
\frac{1}{u'}
\sum_{\substack{x^2/(\log x)^{2\Upsilon}\ll d'\ll x^2/L \\ \gcd(d',W)=1}}
\sum_{
\substack{\fd\in \cP\\ 
\n\fd=d'u'
\\
|\b{v}(d')|\geq  x/(\log x)^\Upsilon
}}
\frac{1}{|\b{v}(d')|}\\
&\ll_\Upsilon \log\log x
\sum_{\substack{u'\ll (\log x)^{2\Upsilon}}}
\frac{1}{u'}
\sum_{\substack{
\v
\in \ZZ^2 \\
 |\b{v}|\ll x/\sqrt{L}}}
\frac{1}{|\v|}
\sum_{\substack{x^2/(\log x)^{2\Upsilon}\ll d'\ll x^2/L \\ \gcd(d',W)=1}}
\sum_{
\substack{\fd\in \cP\\ 
\n\fd=d'u'
\\
bv_1\equiv kv_2\bmod{d'}}}1,
\end{align*}
where we recall that $k$ depends on $\fd$ and $\fe$.  
For any 
  $\fd\in \cP$ with $\n\fd=d'u'$ and $\gcd(\n\fd,W)=1$, there is a factorisation
 $\fd=\fd_1\fd_2$ with 
  $\fd_1,\fd_2\in \cP$  such that 
 $\n\fd_1=d'$, $\n\fd_2=u'$. 
Hence
$$
\sum_{
\substack{\fd\in \cP\\ 
\n\fd=d'u'
\\
bv_1\equiv kv_2\bmod{d'}}}1\leq r^*(u')
\sum_{
\substack{\fd_1\in \cP\\ 
\n\fd_1=d'
\\
\fd_1\mid (bv_1- \theta v_2)}} 1,
$$
by Lemma \ref{lem:honey}. 
On appealing to 
 \eqref{eq:gwr2} to estimate the $u'$-sum, we are left with the contribution
\begin{align*}
&\ll_\Upsilon (\log\log x)^2
\sum_{\substack{
\v
\in \ZZ^2 \\
|\b{v}|\ll x/\sqrt{L}}}
\frac{1}{|\v|}
\sum_{
\substack{\fd_1\in \cP\\
\fd_1\mid (bv_1- \theta v_2)\\
x^2/(\log x)^{2\Upsilon}\ll \n\fd_1 \ll x^2/L \\ \gcd(\n\fd_1,W)=1
}} 1.
\end{align*} 
We will need to restrict the outer sum to a sum over primitive vectors in order to bring 
Lemma \ref{t:NT} into play. 
Let $h=\gcd(v_1,v_2)$ so that 
$\v=h\w$ for  $\w\in \Zp^2$.
Then 
$(bv_1- \theta v_2)
=
(h)(bw_1- \theta w_2)$, 
where $(h)$ is the principal ideal  
generated by $h$. By  unique factorisation, we have 
$\fd_1 \mid (h)(bw_1- \theta w_2)$ if and only if 
$$
\mathfrak{f}^{-1}\fd_1  \mid (bw_1- \theta w_2),
$$ 
where
$\mathfrak{f}$ is defined to be the greatest common ideal divisor of $\fd_1$ and $(h)$.
Writing $\fc=\mathfrak{f}^{-1}\fd_1$, we see that 
\[ 
\sum_{
\substack{\fd_1\in \cP\\
\fd_1\mid (bv_1- \theta v_2)\\
x^2/(\log x)^{2\Upsilon}\ll \n\fd_1 \ll x^2/L \\ \gcd(\n\fd_1,W)=1
}} 1
\leq  
\sum_{\substack{\mathfrak{f}\in \mathcal{P}\\ \mathfrak{f}\mid (h)\\
\gcd(\n 
\mathfrak{f},W)=1}}
\sum_{
\substack{\fc \in \cP\\
\fc \mid (bw_1- \theta w_2)\\
\frac{x^2}{(\log x)^{2\Upsilon}\n \mathfrak{f} }\ll \n\fc \ll \frac{x^2}{L \n \mathfrak{f}}
\\ \gcd(\n\fc,W)=1
}}1. 
\] 
Splitting into $\mathrm{e}$-adic intervals the inner sum
is easily seen to be 
$$\ll_\Upsilon ( \log\log x) \Delta((bw_1-\theta w_2)_W),
$$ 
where
$\Delta(\cdot)=
\Delta(\cdot,\mathbf{1})$, in the notation of 
~\S\ref{s:twist}.
Since there are at most $r^*(h)$ ideals $\mathfrak{f}\in \mathcal{P}$ such that $\mathfrak{f}\mid (h)$ and $\gcd(\n 
\mathfrak{f},W)=1$, we are left with the final contribution 
\begin{align*}
&\ll_\Upsilon (\log\log x)^3
\sum_{h} \frac{r^*(h)}{h}
\sum_{\substack{
\w
\in \Zp^2 \\
|\b{w}|\ll x/(h\sqrt{L})}}
\frac{\Delta((bw_1-\theta w_2)_W)}{|\w|}.
\end{align*}
Splitting into dyadic intervals, we now apply Lemma \ref{t:NT} with $G=\ZZ^2$, combined with 
part (i) of  Lemma \ref{lem:sofos}.  
Noting that one can take $\ve_1>0$ in Lemma~\ref{t:NT} to be arbitrarily small,
we deduce that the sum over $\w$ can be bounded by 
$$
\ll_\epsilon (\log x)^{\epsilon/2}
 \frac{x}{h\sqrt{L}}
$$
for any $\epsilon>0$. This leads to the overall bound
\begin{align*}
&\ll_{\epsilon,\Upsilon} \frac{x(\log x)^\epsilon}{\sqrt{L}}
\sum_{h} \frac{r^*(h)}{h^2} \ll_{\epsilon,\Upsilon} \frac{x(\log x)^\epsilon}{\sqrt{L}},
\end{align*}
which thereby completes the proof of 
 \eqref{eq:daniel}.

\section{The upper bound}
\label{s:epikoskafes}

This section is concerned with proving the upper bound in Theorem \ref{th:upper}. 
Let $X$ be a  quartic del Pezzo surface 
defined over $\Q$, containing a conic defined over $\Q$.
We continue to follow the convention that all implied constants are allowed to depend in any way upon the surface $X$. 

We appeal to \cite[Thm.~5.6 and Rem.~5.9]{FS}.  This shows that there are binary quadratic forms $q_{1,1}^{(i)},q_{1,2}^{(i)},q_{2,2}^{(i)}\in \ZZ[s,t],$ 
for $i=1,2$, 
such that \begin{equation}\label{eq:N1} N (B) \leq \sum_{i=1,2} \sum_{\substack{(s,t ) \in \ZZp^2  \\  |s|,|t| \ll \sqrt{B} \\ \Delta^{(i)}(s,t)\neq 0} } \#\left\{\y\in \ZZp^3:  Q_{s,t}^{(i)}(\y)=0, ~\|\y\|_{s,t}\ll B\right\},\end{equation}where $ \|\y\|_{s,t} = \max\{|s|,|t|\}\max\{|y_1|, |y_2|\}$ and 
$$Q_{s,t}^{(i)}(\y)=q_{1,1}^{(i)}(s,t)y_1^2+q_{1,2}^{(i)}(s,t)y_1y_2+q_{2,2}^{(i)}(s,t)y_2^2+y_3^2.
$$ 
Moreover,  the discriminant $\Delta^{(i)}(s,t)$ of $Q^{(i)}_{s,t}$ is a separable quartic form.  The indices $i=1,2$ are related to the existence of the two complimentary conic bundle fibrations. The two cases $i=1,2$ are treated identically and we shall therefore find it convenient to suppress the index $i$ in the notation. 
It is now clear that we will need a good upper bound for the number of rational points of bounded height on a conic, which is uniform in the coefficients of the defining equation,  a topic that was   addressed in \S \ref{s:sausage}.

\subsection{Application of the bound for conics}

Returning to \eqref{eq:N1}, we 
 apply 
Lemma~\ref{detectors 2} to estimate the inner cardinality. 
For any $(s,t)\in \ZZp^2$, 
an argument of  Broberg \cite[Lemma 7]{broberg} shows that 
 $D_{Q_{s,t}}=O(1)$. 
In our work $W$ is given by \eqref{eq:WW}, with $\nu=1$ and 
$w$  a large parameter depending only on $X$,
  which we will need to 
enlarge at various stages of the argument. In the first instance, we assume that 
$2D_{Q_{s,t}}<w\ll 1$. 
We   deduce that 
\begin{align*} 
N (B) \ll
\sum_{\substack{(s,t ) \in \ZZp^2  \\  |s|,|t| \ll \sqrt{B} \\ \Delta(s,t)\neq 0} } 
C(Q_{s,t},w)
\left(
1+\frac{B}
{|\Delta(s,t)|^{\frac{1}{3}}
\max\{|s|,|t|\}^{\frac{2}{3}}}
\right),
\end{align*}
for any $w>0$, where
$$
C(Q_{s,t},w)
\ll 
\prod_{\substack{p^\xi \| \Delta(s,t) \\ p\leq w}} \tau(p^\xi)
\prod_{\substack{p^\xi \| \Delta(s,t) \\ p>w}}
\l(\sum_{k=0}^\xi
\chi_{Q_{s,t}}(p)^k\r).
$$
Since 
$s, t\ll \sqrt{B}$ and  $\deg(\Delta)=4$, we see that 
\[
|\Delta(s,t)|^{\frac{1}{3}}
\max\{|s|,|t|\}^{\frac{2}{3}}
\ll
\max\{|s|,|t|\}^2\ll B
,\]
whence
\[
1+\frac{B}
{|\Delta(s,t)|^{\frac{1}{3}}
\max\{|s|,|t|\}^{\frac{2}{3}}
}
\ll
\frac{B}
{|\Delta(s,t)|^{\frac{1}{3}}
\max\{|s|,|t|\}^{\frac{2}{3}}
}
.\]

Now let 
\begin{equation}\label{eq:ham}
\Delta(s,t)=\prod_{i=1}^n
\Delta_i(s,t)
\end{equation}
be the factorisation of $\Delta(s,t)$ into irreducible factors over $\Q$.
Each $\Delta_i$ is separable and 
$\res(\Delta_i,\Delta_j)\neq 0$, whenever $i\neq j$.
We suppose that $X$ has $\delta_0=m$ split degenerate fibres and we re-order the factorisation of $\Delta(s,t)$ in such a way that the split degenerate fibres correspond to the closed points
$\Delta_1(s,t),\dots,\Delta_m(s,t)$, with the non-split fibres corresponding to the closed points
$\Delta_{m+1}(s,t),\dots,\Delta_n(s,t)$.
We enlarge $w$ so that 
$$
w>\max_{i\neq j}|\res(\Delta_i,\Delta_j)|.
$$
Loughran, Frei and  Sofos \cite[Part (5) of Lemma~4.8]{FS} 
have  shown that 
 for each $i>m$
there exists a binary form 
$G_i(s,t)\in \ZZ[s,t]$ of even non-negative degree, with 
$\res(G_i,\Delta_i)$ non-zero,
such that 
\[
\chi_{Q_{s,t}}(p)=
\l(
\frac{G_i(s,t)}{p}\r),
\]
for all   $(s,t)\in \ZZp^2$
with $\Delta(s,t)\neq 0$,  
and all
primes $p>w$ with 
$p\mid  \Delta_i(s,t)$.

We proceed by introducing the arithmetic functions
\begin{equation}\label{eq:offal}
\tau_0(s,t)=
\sum_{\substack{d\mid \Delta(s,t)\\ d\mid W^\infty}}1, 
\qquad 
\tau_i(s,t)=
\sum_{\substack{d\mid \Delta_i(s,t) \\ \gcd(d,W)=1}}
1, \quad (1\leq i\leq m),
\end{equation}
and 
\begin{equation}\label{eq:tripe}
r_i(s,t)=
\sum_{\substack{d\mid \Delta_i(s,t) \\ \gcd(d,W)=1}}
\l(
\frac{G_i(s,t)}{d} 
\r), \quad
(m< i\leq n).
\end{equation}
We put
\begin{equation}\label{eq:bacon}
\fS(s,t)=
\tau_0(s,t)
\prod_{i=1}^m
\tau_i(s,t)
\prod_{i=m+1}^n
r_i(s,t), 
\end{equation}
for any $(s,t) \in \ZZp^2$. Note that $\fS(s,t)\geq 0$.  Our work so far shows that 
$$
N(B) \ll B
\sum_{\substack{(s,t ) \in \ZZp^2  \\  |s|, |t| \ll  \sqrt{B} \\ \Delta(s,t)\neq 0} } 
\frac{\fS(s,t)}
{|\Delta(s,t)|^{\frac{1}{3}}
\max\{|s|,|t|\}^{\frac{2}{3}}  
}.
$$ 
Since we are only interested in coprime integers $s,t$, 
there is a satisfactory  contribution of $O(B)$ to the right hand side from   
those vectors $(s,t)$ in which one of the components is zero.   
Hence, by symmetry, 
Theorem~\ref{th:upper} will follow from a bound of the shape 
\begin{equation}
\label{salberger}
\sum_{\substack{(s,t ) \in \ZZp^2\\ 1 \leq |s|\leq |t|\leq \sqrt{B}\\
\Delta(s,t)\neq 0}}
\frac{ 
\fS(s,t)
}
{|\Delta(s,t)|^{\frac{1}{3}}
|t|^{\frac{2}{3}}
}
\ll
(\log B)^{m+1},
\end{equation} 
since \eqref{eq:rank} implies that
$m+1=\rho-1$.

\subsection{Reduction to divisor sums}
For 
$\beta\in \CC$ and $x,y>0$ we let
\[
\c{V}=\left\{
(s,t) \in \R^2:
1\leq |s| \leq |t| \leq x,~
|s-\beta t | \leq y,~
\Delta(s,t) \neq 0
\right\}.
\] 
Consider the divisor function
\begin{equation}
\label{def:erdos}
D_\beta(x,y)=
\sum_{
(s,t) \in \c{V}\cap \ZZp^2
} 
\fS(s,t),
\end{equation}
where $\fS(s,t)$ is given by \eqref{eq:bacon}.
In this subsection we shall establish~\eqref{salberger} subject to the following bound for 
$D_\beta(x,y)$, whose proof will occupy the remainder of the paper. 

\begin{proposition}
\label{prop:nair 1}
Let $\beta \in \CC$, let $\eta\in (0,1)$
and  
assume that 
$x^\eta
\leq
y \leq x$.
Then 
$
D_\beta(x,y) 
\ll_{\beta,\eta}
xy
\l(\log x\r)^{m}.
$
\end{proposition} 

We proceed to show how \eqref{salberger} follows from Proposition \ref{prop:nair 1}.
Since  $\Delta(s,t)$ is separable,  it may  contain the polynomial factor $t$ at most once.
Therefore there exists
$c_0 \in \Q^*$
and pairwise unequal 
$\alpha_i,\alpha_j \in \overline{\Q}$
such that $\Delta(s,t)$ admits the factorisation
$
c_0 t
\prod_{i=1}^3
(s-\alpha_i t)$
or
$c_0 
\prod_{i=1}^4
(s-\alpha_i t)$
,
according to whether $t\mid \Delta(s,t)$ or not, respectively. 
Putting
\begin{equation}
\label{def:alpha}
\alpha
=
\frac{1}{2}
\min_{\substack{i,j,k\\ i\neq j}}\left\{
|\alpha_i-\alpha_j|, |\alpha_k|\right\}
,
\end{equation}
the set of integer pairs $(s,t)$ appearing in~\eqref{salberger}
can be partitioned according to whether or not $(s,t)$ belongs to the set
\[\c{A}=
\left\{
(s,t) \in \R^2:|s-\alpha_i t| \geq \alpha |t|,
\ \text{for all $i$}
\right\}.
\] 
If $(s,t)\in \cA$ then 
$\Delta(s,t)\gg |t|^4$ and it follows that
$$
\sum_{\substack{(s,t ) \in \cA\cap \ZZp^2\\ 1 \leq |s|\leq |t|\leq \sqrt{B}\\
\Delta(s,t)\neq 0}}
\frac{ 
\fS(s,t)
}
{|\Delta(s,t)|^{\frac{1}{3}}
|t|^{\frac{2}{3}}
}
\ll 
\sum_{\substack{(s,t ) \in \cA\cap \ZZp^2\\ 1 \leq |s|\leq |t|\leq \sqrt{B}\\
\Delta(s,t)\neq 0}}
\frac{ 
\fS(s,t)
}
{|t|^{2}
}.
$$
Breaking into dyadic intervals $T/2<|t|\leq T$ and applying Proposition \ref{prop:nair 1} with $x=y=T$ and $\beta=0$, we readily find that the right hand side is $O((\log B)^{m+1})$, which is satisfactory for \eqref{salberger}.

It remains to consider the contribution
to~\eqref{salberger} 
 from  $(s,t)\in \ZZp^2\setminus \cA$.
For each  $i$ we define
$$
S_i(B)=
\sum_{\substack{(s,t ) \in  \ZZp^2\\ 1 \leq |s|\leq |t|\leq \sqrt{B}\\
\Delta(s,t)\neq 0\\
|s-\alpha_i t|<\alpha |t|
}}
\frac{ 
\fS(s,t)
}
{|\Delta(s,t)|^{\frac{1}{3}}
|t|^{\frac{2}{3}}
}.
$$
It now suffices to prove $S_i(B)=O((\log B)^{m+1})$ for each $i$ and each $\alpha_i$.
If $(s,t)$ is counted by $S_i(B)$ then~\eqref{def:alpha} implies that 
for any $j\neq i$
we have 
\[
|s-\alpha_jt| \geq \frac{1}{2}
|\alpha_i-\alpha_j|
|t|,
\]
thus
$|\Delta(s,t)| \gg |t|^3 |s-\alpha_it|$ in $S_i(B)$.
Likewise, we obviously have the reverse inequality  
$|\Delta(s,t)| \ll |t|^3 |s-\alpha_it|$.

We begin by dealing with the contribution of pairs
$(s,t)$ with 
$|s-\alpha_i t|\geq 1$.
For given $S,T$ satisfying 
$1\leq S\ll T\ll \sqrt{B}$, 
the overall  contribution to $S_i(B)$ from 
elements $s,t$ such that 
$T/2<|t|\leq T$ and $S/2<|s-\alpha_i t|\leq S$ is seen to be
$$
\ll \frac{1}{S^{\frac{1}{3}} T^{\frac{5}{3}}} D_{\alpha_i} (T,S),
$$
in the notation of \eqref{def:erdos}. If $S\gg T^{\frac{1}{10}}$ then Proposition \ref{prop:nair 1} shows that this is
$$
\ll \frac{S^{\frac{2}{3}}(\log B)^{m}}{ T^{\frac{2}{3}}}.
$$
Summing over dyadic $S,T$ satisfying $T^{\frac{1}{10}}\ll S\ll T\ll \sqrt{B}$ gives an overall contribution 
$O((\log B)^{m+1})$.
On the other hand, if $S\ll T^{\frac{1}{10}}$, we take $\fS(s,t)\ll T^{\ve}$ for any $\ve>0$,  by the standard estimate for the divisor function, so that $D_{\alpha_i} (T,S)\ll ST^{1+\ve}$. Taking $\ve=\frac{1}{30}$, 
we therefore arrive at the contribution
$$
\ll \frac{S^{\frac{2}{3}}T^{\frac{1}{30}}}{ T^{\frac{2}{3}}}
\ll T^{-\frac{2}{3}+\frac{1}{10}},
$$
from this case.
Again, summing over 
dyadic $S,T$ satisfying $S\ll T^{\frac{1}{10}}$ and $1\ll T\ll \sqrt{B}$, this  shows that we have  an overall contribution 
$O(1)$, which is plainly satisfactory.

It remains to consider the contribution to $S_i(B)$ from integers $s,t$ for which 
$
|s-\alpha_i t| <1
$.
In fact for irrational $\alpha_i$ there are infinitely many 
pairs of coprime integers $s,t$ for which 
$|s-\alpha_i t|<|t|^{-1}$.
The  divisor bound
gives
$\fS(s,t)\ll |t|^{\frac{1}{10}}$,
which leads to 
the contribution 
\begin{equation}
\label{eq:towards1}
\ll 
\sum_{\substack{(s,t ) \in  \ZZp^2, ~\Delta(s,t)\neq 0\\ 1 \leq |s|\leq |t|\leq \sqrt{B}\\
|s-\alpha_i t|< 1
}}
\frac{ 
1
}
{|s-\alpha_i t|^{\frac{1}{3}}
|t|^{\frac{5}{3}-\frac{1}{10}}
}
\end{equation} 
to $S_i(B)$.
We now invoke a result of Davenport and Roth~\cite[Cor.~2]{wow-much-math},
which shows that $\#\mathcal{L}=O(1)$, where
$$
\c{L} =\left\{(s,t)\in \Zp^2:
\left|\alpha_i-\frac{s}{t}\right|<\frac{1}{ |t|^{2+\frac{1}{100}}}
\right\}.
$$
Moreover, the implied constant is effective and  only depends  on the coefficients 
of $\Delta(s,t)$. The contribution to 
\eqref{eq:towards1} from $\c{L}$ is therefore seen to be
$$
\sum_{\substack{(s,t ) \in  \c{L}, ~\Delta(s,t)\neq 0\\ 1 \leq |s|\leq |t|
}}
\frac{ 
1
}
{|s-\alpha_i t|^{\frac{1}{3}}
|t|^{\frac{5}{3}-\frac{1}{10}}
}\ll \c{L} \ll 1, 
$$
since $|s-\alpha_it|\gg |\Delta(s,t)| |t|^{-3}\gg |t|^{-3}$. On the other hand, the contribution to \eqref{eq:towards1}
 outside of $\c{L}$ is
$$
\ll 
\sum_{\substack{(s,t ) \in  \ZZp^2\setminus \c{L}\\ 1 \leq |s|\leq |t|\leq \sqrt{B}\\ |s-\alpha_i t|<1
}}
\frac{1}{|t|^{\frac{4}{3}-\frac{1}{10}-\frac{1}{300}}}
\ll 
\sum_{|t|\leq \sqrt{B}}
\frac{1}{|t|^{\frac{4}{3}-\frac{1}{10}-\frac{1}{300}}} \ll 1,
$$
since for given $t$ there are finitely many integers $s$ in the interval $|s-\alpha_i t|<1$.
This completes the deduction of \eqref{salberger} from Proposition~\ref{prop:nair 1}.

\subsection{Small divisors}

The function $\tau_0(s,t)$ in \eqref{eq:bacon} is concerned with the contribution to $\fS(s,t)$ from small primes 
 $p\leq w$. Our work in \S \ref{s:sausage} only applies to divisor sums supported away from small prime divisors. Hence we shall  begin by using the geometry of numbers to deal with the function $\tau_0(s,t)$, before handling the remaining factors  in $\fS(s,t)$.

Following Daniel \cite{daniel}, for any $a \in \N$
we call two vectors $\b{x} ,\b{y}\in \ZZ^2$ equivalent modulo $a$
if
\[
\gcd(\b{x},a)=
\gcd(\b{y},a)=1
\quad
\text{ and }
\quad
\Delta(\b{x})\equiv \Delta(\b{y})\equiv 0 \bmod{a},
\] 
and, moreover, there exists $\lambda \bmod{a}$ such that
$\b{x}\equiv \lambda \b{y} \bmod{a}$. 
The set of equivalence classes is denoted by $\fA(a)$
and the class elements as $\c{A}$.
Letting
\[
\varrho^*(a)=\#\left\{
(\sigma,\tau) \bmod{a}:
\gcd(\sigma,\tau,a)=1,~
\Delta(\sigma,\tau) \equiv 0 \bmod{a}
\right\},
\]
we find that 
$
\varrho^*(a)=
\phi(a)
\#\fA(a).
$
Moreover, we clearly have $$\varrho^*(a)\leq \phi(a)(
\rho_{\Delta(x,1)}(a)+
\rho_{\Delta(1,x)}(a)),$$
in the notation of \eqref{eq:def-rho}.
Since   $\Delta(s,t)$ is separable, 
it follows from  Huxley~\cite{huxley} that 
$\rho_{\Delta(x,1)}(a)\leq 4^{\omega(a)}|\disc(\Delta)|^{\frac{1}{2}}$, 
and similarly for 
$\rho_{\Delta(1,x)}(a)$. Hence
\begin{equation}
\label{eq:huxley}
\#\fA(a)=
\frac{\varrho^*(a)}{
\phi(a)}
\ll
4^{\omega(a)}.
\end{equation}

For each $(s,t)\in \c{V}\cap \ZZp^2$, write
\[
r(s,t)=
\prod_{i=1}^m
\tau_i(s,t)
\prod_{i=m+1}^n
r_i(s,t)
.\]
Then 
\begin{align*}
D_\beta(x,y)
&\leq 
\sum_{\substack{q \ll x^{4}\\ q \mid W^\infty }}
\sum_{\substack{(s,t) \in \c{V}\cap 
\ZZp^2\\ q \mid\Delta(s,t)}}
r(s,t)\\
&\leq 
\sum_{\substack{q \ll x^{4}\\ q \mid W^\infty }}
\sum_{\cA\in \mathfrak{A}(q)}
\sum_{\substack{(s,t) \in \c{V}\cap G(\cA)\cap\ZZp^2}}
r(s,t),
\end{align*}
where $G(\c{A})=\{\x\in \ZZ^2: \exists \lambda\in \ZZ~\exists\y\in \cA \text{ s.t. } \x\equiv \lambda \y\bmod{q}\}$ is the lattice  generated by the vectors in 
$\c{A}$. The determinant of this lattice is $q$.
We shall establish the  following result.

\begin{proposition}
\label{prop:daniel}
Let $\eta\in (0,1)$
and assume that 
$x^\eta \leq y \leq x$.
Then 
\[
\sum_{\substack{(s,t) \in \c{V}\cap G(\cA)\cap\ZZp^2}}
r(s,t) \ll_{\beta,\eta,N}
xy
\l(\frac{(\log x)^m}{q}+
\frac{1}{(\log x)^{N}}
\r)
,\]
for any $N>0$,
where the implied constant is independent of $q$.
\end{proposition}

We now show how  Proposition \ref{prop:nair 1} follows from this result. 
Employing  \eqref{eq:huxley}, we deduce that 
\[
D_\beta(x,y)
\ll_{\beta, \eta,N}
xy
(\log x)^m 
\sum_{\substack{q \ll x^{4}\\ q \mid W^\infty }}
\frac{ 4^{\omega(q)}}{q}
+
\frac{xy}{(\log x)^{N}}
\sum_{\substack{q \ll x^{4}\\ q \mid W^\infty }}
4^{\omega(q)}.
\]
The first sum is $
\ll
(\log w)^{4}
\ll 1$.  On the other hand, the second sum is
\[
\leq
\prod_{p \leq w} \l(16\log x+O(1)\r)
\ll
(\log x)^{\pi(w)}.
\] Choosing $N=\pi(w)$, we therefore conclude the deduction of
Proposition~\ref{prop:nair 1} from  Proposition~\ref{prop:daniel}.

\subsection{The final push}

The aim of this subsection is to prove Proposition~\ref{prop:daniel}.
Recall from \eqref{eq:ham} that we have a factorisation 
\[
\Delta(s,t)=\prod_{i=1}^m \Delta_i(s,t)
\prod_{i=m+1}^n \Delta_i(s,t),
\]
where each $\Delta_i\in \ZZ[s,t]$ is irreducible  and 
the fibre above the closed point $\Delta_i$  is split if and only if $i\leq m$. 
We  now want to bring into play the work in \S \ref{s:sausage}, in order to transform the sum in Proposition \ref{prop:daniel} into one that can be handled by Lemma  \ref{t:NT}.

Let $i\in \{1,\dots,n\}$.  Recall from \eqref{eq:offal} and \eqref{eq:tripe} that we are interested in the divisor sum
$$
\sum_{\substack{d\mid \Delta_i(s,t) \\ \gcd(d,W)=1}}
\l(
\frac{G_i(s,t)}{d}
\r),
$$
where $G_i(s,t)\in \ZZ[s,t]$ is a  form  of even degree (and we allow $G_i(s,t)$ to be identically equal to $1$).
This is exactly of the form considered in \eqref{eq:def-hW}.
Let  $b_i=\Delta_i(1,0)\in \ZZ$ and  suppose for the moment that $b_i\neq 0$. 
As previously, let $\theta_i$ be a root of the polynomial $\tilde\Delta_i(x,1)$, in the notation of 
\eqref{eq:burger}, and write $K_i=\QQ(\theta_i)$.  Let $\fo_i$ denote the ring of integers of $K_i$. We enlarge $w$ to ensure that $w>2b_i D_{L_i/K_i} \Delta_{\theta_i}$, where 
$\Delta_{\theta_i}$ is given by \eqref{eq:tesco} and 
$L_i=K_i(\sqrt{G_i(b_i^{-1}\theta_i,1)})$.  
Thus 
$$
[L_i:K_i]=\begin{cases}
1 &\text{ if $i\leq m$,}\\
2 &\text{ if $i> m$.}
\end{cases}
$$
Next, let $\psi_i$ be the quadratic Dirichlet  character constructed in \S \ref{s:sausage} (taking $\psi_i=1$ when $G_i(s,t)$ is identically $1$).
Let $\n_i$ denote the ideal norm in $K_i$. 
Then it follows from part (iii) of Lemma \ref{lem:reuss} that 
for any $(s,t)\in \ZZp^2$ such that $\Delta_i(s,t) \neq 0$, we have 
\begin{equation}\label{eq:parsons nose}
\sum_{\substack{d\mid \Delta_i(s,t) \\ \gcd(d,W)=1}}
\l(
\frac{G_i(s,t)}{d}
\r)=
\sum_{\substack{\fa\mid (b_is-\theta_i t)\\ \gcd(\n_i\fa,W)=1}} \psi_i(\fa).
\end{equation}
Moreover, 
if  $\cP_i^\circ$, $
\cP_i$ are defined  
as in \eqref{eq:span-0} and \eqref{eq:span}, respectively,    
then part (i) of Lemma \ref{lem:reuss} implies that $\fa\in \cP_i$ for any 
$\fa\mid (b_is-\theta_i t)$ such that $\gcd(\n_i\fa,W)=1$.

Suppose now that $b_i=0$, so that  $\Delta_i(s,t)=c t$ for some non-zero $c \in \ZZ$. We enlarge 
 $w$ to ensure that $w>c$.  In this case we have 
$$
\sum_{\substack{d\mid \Delta_i(s,t) \\ \gcd(d,W)=1}}
\l(
\frac{G_i(s,t)}{d}
\r)
=
\sum_{\substack{d\mid  t\\ \gcd(d,W)=1}}
\l(
\frac{G_i(s,t)}{d}
\r)
=
\sum_{\substack{d\mid  t\\ \gcd(d,W)=1}}
\l(
\frac{G_i(1,0)}{d}
\r),
$$
since $G_i$ has even degree and $(s,t)\in \ZZp^2$.
But this is of the shape  \eqref{eq:parsons nose}, with 
$b_i=0$, $\theta_i=1$, $K_i=\Q$,  and 
$\psi_i(d)=(
\frac{G_i(1,0)}{d}
)$.

Let $i\in \{1,\dots,n\}$ and let  $\fc\subset \fo_i$ be an integral ideal. 
We define multiplicative functions $\ft_i,\fr_i\in \c{M}_{K_i}$, in the notation of \S \ref{s:Mk},
via 
$$
\ft_i(\fc)=
\sum_{
\substack{\fa\in \cP_i\\ 
 \fa\mid \fc}}
1, \quad (1\leq i\leq m),
$$
and 
$$
\fr_i(\fc)=
\sum_{
\substack{\fa\in \cP_i\\ 
 \fa\mid \fc}}
\psi_i(\fa),
\quad (m< i\leq n).
$$
It follows that 
$$
r(s,t) =
\prod_{i=1}^m
\ft_{i,W}(b_i s-\theta_i t)
\prod_{i=m+1}^n
\fr_{i,W}(b_i s-\theta_i t)
$$
in Proposition~\ref{prop:daniel}, for 
any $(s,t)\in \ZZp^2$.

We are now in a position to apply Lemma \ref{t:NT} with $\c{R}=\c{V}$,
$G=G(\c{A})$ and $q_G=q$.
In particular 
it follows that 
$$
xy\ll V=\vol(\c{R}) \ll xy \quad \text{ and } \quad 
x \log x\ll K_{\c{R}}\ll  x \log x.
$$
According to the statement of Proposition \ref{prop:daniel}, we are given $\eta\in (0,1)$ and 
$x,y$ such that $x^\eta\leq y\leq x$. Thus  $\c{R}$ is regular. Since $q\ll x^4$, it therefore follows that all the hypotheses of Lemma~\ref{t:NT} 
are met with each $\ve_i>0$ being arbitrarily small.
 On enlarging $w$ suitably, we deduce that 
\begin{align*}
\sum_{\substack{(s,t) \in \c{V}\cap G(\cA)\cap\ZZp^2}}
r(s,t) 
\ll_{\eta,  W}~&
\frac{xy}{(\log x)^n}
\frac{h_W^*(q)}{q}
\prod_{i=1}^m
E_{\ft_i}(x^2;1)
\prod_{i=m+1}^n
E_{\fr_i}(x^2;1)\\
&+x^{1+\frac{\eta}{2}},
\end{align*}
Note  that $h_W^*(q)=1$, since $q\mid W^\infty$.
Moreover, since $x^{1+\frac{\eta}{2}}\ll_N xy(\log x)^{-N}$, for any $N>0$, the second term here is plainly satisfactory for Proposition \ref{prop:daniel}.

Finally, we have 
$$
E_{\ft_i}(z;1)
=\exp\l(
\sum_{\substack{\n_i
\fp \leq z
\\
\fp\in \c{P}_i^\circ 
}}\frac{\ft_i(\fp)}{\n_i\fp}\r)= 
\exp\l(
\sum_{\substack{\n_i
\fp \leq z
\\
\fp\in \c{P}_i^\circ 
}}\frac{2}{\n_i\fp}\r)\ll (\log z)^{2},
$$
for $i\in \{1,\dots,m\}$, 
and
$$
E_{\fr_i}(z;1)
=\exp\l(
\sum_{\substack{\n_i
\fp \leq z
\\
\fp\in \c{P}_i^\circ
}}\frac{\fr_i(\fp)}{\n_i\fp}\r)= 
\exp\l(
\sum_{\substack{\n_i
\fp \leq z
\\
\fp\in \c{P}_i^\circ
}}\frac{1+\psi_i(\fp)}{\n_i\fp}\r)\ll \log z,
$$
for $i\in \{m+1,\dots,n\}$.
Thus the first term makes the overall contribution
$$
\ll
\frac{xy(\log x)^{m}}{q},
$$
which thereby completes the proof of Proposition \ref{prop:daniel}.

\end{document}